\documentclass[12pt]{amsart}

\usepackage[utf8]{inputenc}
\usepackage{amsmath,mathtools,amsthm,amsfonts,amssymb,verbatim}
\usepackage{color, colortbl} 
\usepackage{enumitem} 
\usepackage{url}
\usepackage{tikz} 
\usepackage{diagbox}
\usetikzlibrary{patterns}
\usetikzlibrary{arrows,shapes,positioning,backgrounds,calc}
\usetikzlibrary{patterns.meta}

\usepackage[letterpaper]{geometry} 
\usepackage{caption}
\usepackage[labelformat=simple,labelfont={}]{subcaption}

\theoremstyle{definition} 
 
\newtheorem{example}[equation]{Example}
\newtheorem{corollary}[equation]{Corollary} 
 
\newtheorem{lemma}[equation]{Lemma}
\newtheorem{proposition}[equation]{Proposition}

\theoremstyle{definition}
\newtheorem{definition}[equation]{Definition}

\newtheorem{remark}[equation]{Remark}

\numberwithin{equation}{section}

\definecolor{mjo}{rgb}{0,0,.9}
\newcommand{\multiset}[1]{\{\!\!\{#1 \}\!\!\}}

\DeclareMathOperator{\Opt}{Opt}
\DeclareMathOperator{\Con}{Con}
\global\long\def\mex{\operatorname{mex}}
\global\long\def\nim{\operatorname{nim}}

\DeclareMathOperator{\fbd}{fbd}

\usepackage{relsize}
\newcommand{\boxprod}{\mathop{\textstyle\mathsmaller{\square}}}

\usetikzlibrary{arrows}
\usetikzlibrary{shapes} 
\usetikzlibrary{patterns}
\tikzstyle{vert} = 
[circle, draw, fill=grey!40,inner sep=0pt, minimum size=5mm]
\tikzstyle{small vert} = 
[circle, draw,fill=grey!40, inner sep=2pt, minimum size=2.7mm]
\tikzstyle{rect vert} = 
[rectangle, rounded corners, draw, fill=grey!40,inner sep=2.7pt, minimum size=4mm]
\tikzstyle{b} = [draw, very thick, black,-]
\tikzstyle{d} = [draw, thick, black,-stealth]
\tikzstyle{a} = [draw, very thick, black,-stealth]

\definecolor{grey}{rgb}{.7, .7, .7}

\definecolor{orange}{RGB}{255,102,0}
\definecolor{ggreen}{RGB}{0,153,0}
\definecolor{darkblue}{RGB}{0,0,255}
\definecolor{purple}{RGB}{153,51,255}
\definecolor{turq}{RGB}{72,209,204}
\definecolor{gray}{RGB}{220,220,220}
\definecolor{orange2}{RGB}{255,100,0}
\definecolor{purple2}{RGB}{159,51,250}
\definecolor{rred}{rgb}{0.9, 0.17, 0.31}
\definecolor{naugreen}{cmyk}{.43,0,.34,.38}
\definecolor{naublue}{cmyk}{1,.72,0,.32}
\definecolor{mediterranean}{cmyk}{.67,0,.08,.3}
\definecolor{rose}{cmyk}{0,1.00,.20,0}
\definecolor{darkorchid}{cmyk}{.6,.9,0,.05}
\definecolor{butterfly}{cmyk}{.95,.59,0,.10}
\definecolor{springgreen}{cmyk}{1.00,0,.70,.02}
\definecolor{darkred}{cmyk}{0,1,1,.5}
\definecolor{nectarine}{cmyk}{0,0.70,1.00,0}
\definecolor{icyblue}{cmyk}{.84,.25,0,.06}
\definecolor{manatee}{rgb}{0.59, 0.6, 0.67}

\title{Categories of impartial rulegraphs and gamegraphs}
\author{Bojan Ba\v{s}i\'{c}}
\address{Department of Mathematics and Informatics, University of Novi Sad, Trg Dositeja Obra\-dovi\'ca 4, 21000 Novi Sad, Serbia}
\email{bojan.basic@dmi.uns.ac.rs}
\author{Paul Ellis}
\address{Department of Mathematics, Rutgers University, 110 Frelinghuysen Rd., Piscataway, NJ 08854-8019, USA}
\email{paulellis@paulellis.org}
\author{Dana C.~Ernst}
\address{Northern Arizona University, Department of Mathematics and Statistics, Flagstaff, AZ 86011-5717, USA}
\email{dana.ernst@nau.edu}
\author{Danijela Popovi\'{c}}
\address{Mathematical Institute of the Serbian Academy of Sciences and Arts, Kneza Mihaila 36, 11000 Belgrade, Serbia}
\email{danijela.mitrovic@dmi.uns.ac.rs}
\author{N\'{a}ndor Sieben}
\address{Northern Arizona University, Department of Mathematics and Statistics, Flagstaff, AZ 86011-5717, USA}
\email{nandor.sieben@nau.edu}

\date{\today}

\begin{document}

\maketitle

\begin{abstract}
The traditional mathematical model for an impartial combinatorial game is defined recursively as a set of the options of the game, where the options are games themselves. We propose a model called gamegraph, together with its generalization rulegraph, based on the natural description of a game as a digraph where the vertices are positions and the arrows represent possible moves. Such digraphs form a category where the morphisms are option preserving maps. We study several versions of this category. Our development includes congruence relations, quotients, and isomorphism theorems and is analogous to the corresponding notions in universal algebra. The quotient by the maximum congruence relation produces an object that is essentially equivalent to the traditional model. After the development of the general theory, we count the number of non-isomorphic gamegraphs and rulegraphs by formal birthday and the number of positions. 
\end{abstract}

\section{Introduction}

Standard development of combinatorial game theory uses a model in which a game is a set of its options and making a move is the selection of an element from this set. The game ends when there are no available moves because the current position is the empty set. This approach results in a well-behaved mathematical object that captures the essence of the game. However, in practice this is not how we describe a game. We usually explain the rules for moving from one position of the game to another. The mathematical object that fits this description is a certain digraph that we call a rulegraph. To play an actual game, we also need a starting position. We call the resulting object a gamegraph.

It is not immediately clear how these two approaches fit together. One resolution is to identify every vertex with the set of its options. This is a recursive process that usually starts at the terminal positions. In particular, we identify all the terminal positions and it becomes the empty set. Then we identify all the positions whose only option is the empty set, and so on. This identification process builds a minimum quotient that is equivalent to the standard definition of a game. Yet during this process a lot of information is lost that is likely useful for human players because it provides intuition. Chess players, for example, learn a large number of mate positions and do not think of the end of the game as the empty set. In fact, libraries are full of books that contain actual chess positions as pieces on the squares of the board.

Formally, we model games as a category where the objects are gamegraphs and the morphisms are option preserving maps. Our approach provides a uniform theory that captures both the practical aspects of playing concrete games, as well as the technical notion of a game. Our model allows for the creation of outcome-preserving quotient gamegraphs including the minimum quotient of the game that is equivalent to the standard definition of a game, as well as intermediary quotients that remove some unimportant information while preserving enough information that guides our intuition and makes the analysis of the game easier. An illustration is Grundy's game, where players divide one of the available heaps into two non-equal parts. As heaps of size $1$ or $2$ are irrelevant for the rest of the game, we can identify all the positions that differ only in such heaps. As we will see in Example~\ref{ex:grundy}, the minimum quotient of this game requires further identifications of positions that is less useful for gaining insight.

Our development mimics that of universal algebra. Indeed, we establish two results that are direct analogs to the First and Fourth Isomorphism Theorems for universal algebras (see Propositions~\ref{prop:fundamental homomorphism theorem} and~\ref{prop:fourth iso theorem}). Moreover, we have a notion of simple that coincides precisely with definition given in the context of universal algebra, namely, a structure is simple if it has no nontrivial congruence relations. One important distinction between our development and that of universal algebra is that we do not have a notion of a trivial quotient where every position is identified since our digraphs cannot be ``collapsed vertically" (see Proposition~\ref{prop:noVerticalCollapse}). Instead, each of our digraphs has a unique minimum quotient (see Proposition~\ref{prop:unique simplification}) that is simple.

A comprehensive treatment of the standard theory of impartial games can be found in~\cite{albert2007lessons,ONAG,SiegelBook}. Conway writes \cite[Chapter 11]{ONAG} the following about impartial games: ``we identify each game with the set of its options". Our paper attempts to spell out the details of this sentence.  

\section{Preliminaries}

A \emph{digraph} $D$ is a pair $(V(D),E(D))$, where $V(D)$ is the set of \emph{vertices} and $E(D)\subseteq V(D)\times V(D)$ is the set of \emph{arrows}. If $S\subseteq V(D)$ is any subset of vertices of $D$, then we define the \emph{subdigraph induced by $S$} to be the digraph whose vertex set is $S$ and whose arrow set consists of all arrows in $E(D)$ that have endpoints in $S$.  For an arrow $(u,v)$ we say that $u$ is the \emph{in-neighbor} of $v$ and that $v$ is the \emph{out-neighbor} of $u$. 
A vertex is called a \emph{source} if it has no in-neighbors and a sink if it has no out-neighbors. 
A \emph{directed walk} in digraph $D$ is a finite or infinite sequence of vertices $v_1,v_2,\dots$ such that $(v_i,v_{i+1})$ is an arrow for all $i\geqslant 1$. 


A \emph{digraph map} $\alpha:C\to D$ between digraphs $C$ and $D$ is a function $\alpha:V(C)\to V(D)$.
A \emph{homomorphism} is a digraph map $\alpha:C\to D$ such that $(u,v)\in E(C)$ implies $(\alpha(u),\alpha(v))\in E(D)$. 
Digraphs and homomorphisms form a concrete category $\mathbf{Gph}$.


Any function $f:X\to Y$ gives rise to an equivalence relation $\sim$, called the \emph{kernel} of $f$, defined on $X$ by $u \sim v$ if and only if $f(u)=f(v)$. Of course, this induces a partition on $X$. We denote the corresponding quotient by $X/f$.

\section{Impartial rulegraphs and gamegraphs} \label{sec:impartial}

In this paper, we focus exclusively on impartial games, so that all games are implicitly assumed to be impartial. We are going to model games with special digraphs. So we will call the vertices of a digraph \emph{positions} and the set $\Opt(p)$ of out-neighbors of a position $p$ the \emph{options} of $p$.

\begin{definition}
A \emph{rulegraph} $\mathsf{R}$ is a digraph with no infinite directed walk. 
\end{definition}

A rulegraph is called \emph{finite} if it has finitely many positions. A position $p$ is called \emph{terminal} if $\Opt(p)=\emptyset$. We say that $q$ is a \emph{subposition} of $p$ if there is a directed walk from $p$ to $q$.  Note that we allow walks of length zero, so that every position is a subposition of itself.

\begin{definition}
A \emph{gamegraph} $\mathsf{G}$ is a rulegraph that has a unique source vertex called the \emph{starting position} and every position is a subposition of the starting position.
\end{definition}

During a \emph{play} on the gamegraph $\mathsf{G}$, two players take turns replacing the current position with one of its options. At the beginning of the play the current position is the starting position. The play ends when the current position becomes a terminal position. Thus, a play is essentially a directed walk from the starting position to one of the terminal positions. 

To determine the winner at the end of a play we need a winning condition. This amounts to classifying each terminal position as winning or losing. Commonly used such assignments are \emph{normal play} where the player who moves last wins, and \emph{mis\`{e}re play} where the player who moves last loses. The choice of the winning condition determines for each position which player has a winning strategy starting at that position. Such a choice determines an \emph{outcome function} $o:V(\mathsf{R})\to \{P,N\}$ on the positions of a rulegraph $\mathsf{R}$, where $P$ and $N$ indicate whether the previous or the next player wins from position $p$. The outcome of $\mathsf{G}$ is $o(\mathsf{G}):=o(s)$, where $s$ is the starting position of $\mathsf{G}$.

\begin{proposition}
\label{prop:reachable}
Every position of a finite rulegraph $\mathsf{R}$ with a unique source is a subposition of the starting position.
\end{proposition}

\begin{proof}
Let $p$ be a position of $\mathsf{R}$. Traveling backward along the arrows, we can find a maximal sequence $q_0,q_1,q_2,\ldots$ of positions such that $q_0=p$ and $q_n\in\Opt(q_{n+1})$ for all $n$. The sequence does not have any position more than once since the rulegraph has no infinite walks. Hence the sequence must be finite since the rulegraph is finite. The last position in the sequence must be the unique source since the sequence is maximal.  
\end{proof}

\begin{remark}
A consequence is that a finite rulegraph with a unique source is automatically a gamegraph. We do not need to check that every position is a subposition of the starting position.
\end{remark}

\begin{example}
The previous proposition is not true without the finiteness assumption. The rulegraph depicted in Figure~\ref{fig:ruleset not game} has a unique source but only two positions are subpositions of the starting position. Thus, this rulegraph is not a gamegraph. 
\end{example}

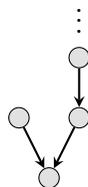
\begin{figure}[ht]
\centering
\begin{tikzpicture}[xscale=.4,yscale=.8,auto]
\node (1) at (3,2) [small vert] {};
\node (2) at (1,1) [small vert] {};
\node (3) at (2,0) [small vert] {};
\node (4) at (3,1) [small vert] {};
\node at (3,2.75) {$\vdots$};
\path [d] (1) to (4);
\path [d] (2) to (3);
\path [d] (4) to (3);
\end{tikzpicture}
\caption{
\label{fig:ruleset not game}
A rulegraph with a unique source that is not a gamegraph.
}
\end{figure}

Since a rulegraph $\mathsf{R}$ has no infinite walks, no position is a proper subposition of itself. Each position $p$ of $\mathsf{R}$ determines a gamegraph $\mathsf{G}=\mathsf{R}_p$ which is the subdigraph of $\mathsf{R}$ induced by the subpositions of $p$. If $p$ is the initial position of a gamegraph $\mathsf{G}$, then $\mathsf{G}_p$ is of course $\mathsf{G}$. 

For a rulegraph $\mathsf{R}$, let $\Gamma(\mathsf{R}):=\{\mathsf{R}_p\mid p\in V(\mathsf{R})\}$. We say that $\mathsf{R}$ is a \emph{rulegraph for} a collection $\mathcal{G}$ of gamegraphs if each $\mathsf{G}$ in $\mathcal{G}$ is isomorphic to $\mathsf{R}_p$ as digraphs for some $p$. Note that $\mathsf{R}$ is a rulegraph for the collection $\Gamma(\mathsf{R})$, and the disjoint union $\bigsqcup \mathcal{G}$ is a rulegraph for $\mathcal{G}$.

The \emph{minimum excludant} $\mex(A)$ of a set $A$ of ordinals is the smallest ordinal not contained in the set. The \emph{nim-number} $\nim(p)$ of a position $p$ is the minimum excludant of the set of nim-numbers of the options of $p$. That is, 
\[
\nim(p):=\mex(\nim(\Opt(p)).
\]
The minimum excludant of the empty set is $0$, so the terminal positions of a gamegraph have nim-number $0$. 
The \emph{nim-number} $\nim(\mathsf{G})$ of a gamegraph $\mathsf{G}$ is the nim-number of its starting position. 

The nim-number of a gamegraph determines the outcome of a gamegraph under normal play  since a position $p$ is a $P$-position if and only if $\nim(p)=0$. 

The \emph{formal birthday of a position} $p$ of a rulegraph $\mathsf{R}$ is the ordinal defined recursively via
\[
\fbd(p):=\sup\{\fbd(q)+1 \mid q\in \Opt(p)\},
\]
with the interpretation that $\sup(\emptyset):=0$, so that $\fbd(p)=0$ if $p$ is a terminal position. The \emph{formal birthday of the rulegraph} $\mathsf{R}$ is $\fbd(\mathsf{R}):=\sup\{\fbd(p) \mid p\in V(\mathsf{R})\}$. Note that in the case of a gamegraph $\mathsf{G}$ with source $s$, we have $\fbd(\mathsf{G})=\fbd(s)$.

The \emph{box product} of two digraphs $C$ and $D$, denoted $C \boxprod D$, is the digraph whose vertex set is $V(C) \times V(D)$ and there is an arrow from $(x_1, y_1)$ to $(x_2, y_2)$ provided either
\begin{enumerate}
 \item $x_1=x_2$ and there is an arrow from $y_1$ to $y_2$ in $D$, or
 \item $y_1=y_2$ and there is an arrow from $x_1$ to $x_2$ in $C$.
\end{enumerate} 

The \emph{sum} $\mathsf{R}+\mathsf{S}$ of the rulegraphs  $\mathsf{R}$ and $\mathsf{S}$ is the product digraph $\mathsf{R}\boxprod\mathsf{S}$. See Figure~\ref{fig:sum of slims} for an example. This means that in each turn a player makes a valid move either in rulegraph $\mathsf{R}$ or in rulegraph $\mathsf{S}$ and a play sum ends when both plays end. The nim-number of the sum of two gamegraphs $\mathsf{G}$ and $\mathsf{H}$ can be computed as the \emph{nim-sum} 
\[
\nim(\mathsf{G}+\mathsf{H})=\nim(\mathsf{G})\oplus\nim(\mathsf{H}),
\]
which requires binary addition without carry. 

\begin{example}
\label{exa:NIM}
For each ordinal $\xi$, we define one-pile NIM with $\xi$ stones to be the gamegraph $\star\xi$ with set of positions $\{0,1,\ldots,\xi\}$ and option function defined by $\Opt(\zeta):=\zeta$. Recall that $\omega$ is the first infinite ordinal. One possible rulegraph for the collection of all finite one-pile NIM gamegraphs is $\bigsqcup_{n\in\omega} \star n$ and another is $\star\omega$ with its starting position $\omega$ removed. In the former rulegraph each game is represented infinitely many times, while in the latter these representatives are unique. The Sprague--Grundy Theorem~\cite{albert2007lessons,SiegelBook} implies that $\nim(\star\xi)=\xi$.
Note that our definition of $\star\xi$ is different from the traditional definition of the nimber $*\xi$ as $\star\xi$ is a digraph.
\end{example}

\section{Option preserving maps}\label{sec:option preserving}

A map $\alpha$ between rulegraphs or gamegraphs is called a \emph{rulegraph map} or \emph{gamegraph map}, respectively. 

\begin{definition}
A digraph map $\alpha:C\to D$ is \emph{option preserving} if  $\Opt(\alpha(p))=\alpha(\Opt(p))$ for each $p\in V(C)$. 
\end{definition}

\begin{example}
Figure~\ref{fig:OptPreserving} shows an option preserving rulegraph map $\alpha:\mathsf{R}\to\mathsf{S}$. Note that $\mathsf{S}$ is actually a gamegraph.  The map $\alpha$ is indicated in the figure using matching colors and shapes.
\end{example}

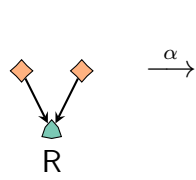
\begin{figure}[ht]
\begin{tikzpicture}[scale=.8]
\node (02) [small vert,fill=orange!50,star,star points=4] at (-1,1) {};
\node (01) [small vert,fill=orange!50,star,star points=4] at (0,1) {};
\node (00) [small vert,fill=springgreen!50,star,star points=3
] at (-.5,0) {};
\path [d] (02) to (00);
\path [d] (01) to (00);
\node (R) at (-.5,-.5) {$\mathsf{R}$};
\end{tikzpicture}
$\quad$
\raisebox{1.4cm}{\hbox{$\overset\alpha\longrightarrow$}}
$\quad$
\begin{tikzpicture}[scale=.8]
\node (a) [small vert] at (3,2) {};
\node (b) [small vert,fill=orange!50,star,star points=4] at (3,1) {};
\node (d) [small vert,fill=springgreen!50,star,star points=3] at (3,0) {};
\path [d] (a) to (b);
\path [d] (b) to (d);
\node (S) at (3,-.5) {$\mathsf{S}$};
\end{tikzpicture}

\caption{
\label{fig:OptPreserving}
An option preserving rulegraph map.}
\end{figure}

\begin{definition}
A gamegraph map $\alpha:\mathsf{G}\to\mathsf{H}$ is \emph{source preserving} if it takes the starting position of $\mathsf{G}$ to the starting position of $\mathsf{H}$.
\end{definition}

Option preserving maps often come from symmetries.

\begin{example}\label{ex:44}
The positions of NIM played on several piles can be modelled using tuples or multisets. The mapping $\alpha:\mathsf{G}\to \mathsf{H}$ defined via $\alpha(a_1,\ldots,a_n)=\multiset{a_1,\ldots,a_n}$ between these two representations is both option and source preserving. This map essentially comes from the permutation symmetries of the tuple representation. The corresponding gamegraphs $\mathsf{G}$ and $\mathsf{H}$ in the special case of two piles consisting of 3 and 2 stones each, are shown in Figure~\ref{fig:ex of morphism}. The map $\alpha$ is indicated in the figure using matching colors.
\end{example}

\begin{figure}[ht]
\centering
\begin{tikzpicture}[scale=1.2,auto]
\node (1) at (0,5) [rect vert,fill=purple!50] {\scriptsize $(3,2)$};
\node (2) at (-1,4) [rect vert,fill=orange!50] {\scriptsize $(2,2)$};
\node (3) at (1,4) [rect vert,fill=turq!50] {\scriptsize $(3,1)$};
\node (4) at (-2,3) [rect vert,fill=springgreen!50] {\scriptsize $(1,2)$};
\node (5) at (0,3) [rect vert,fill=springgreen!50] {\scriptsize $(2,1)$};
\node (6) at (2,3) [rect vert,fill=rose!50] {\scriptsize $(3,0)$};
\node (7) at (-2,2) [rect vert,fill=darkorchid!50] {\scriptsize $(1,1)$};
\node (8) at (0,2) [rect vert] {\scriptsize $(0,2)$};
\node (9) at (2,2) [rect vert] {\scriptsize $(2,0)$};
\node (10) at (-1,1) [rect vert,fill=mediterranean!50] {\scriptsize $(1,0)$};
\node (11) at (1,1) [rect vert,fill=mediterranean!50] {\scriptsize$(0,1)$};
\node (12) at (0,0) [rect vert,fill=rred!50] {\scriptsize $(0,0)$};
\node (G) at (0,-.5) {$\mathsf{G}$};
\path [d,out=196,in=80] (1) to (4);
\path [d] (1) to (2);
\path [d,out=-60,in=45] (1) to (8);
\path [d] (1) to (3);
\path [d,out=-15,in=110] (1) to (6);
\path [d] (2) to (4);
\path [d,out=-90] (2) to (8);
\path [d] (2) to (5);
\path [d,out=-20] (2) to (9);
\path [d,out=200,in=45] (3) to (7);
\path [d] (3) to (5);
\path [d] (3) to (6);
\path [d] (3) to (11);
\path [d,out=-45,in=90] (4) to (10);
\path [d] (4) to (7);
\path [d] (4) to (8);
\path [d] (5) to (7);
\path [d,out=-45,in=90] (5) to (11);
\path [d] (5) to (9);
\path [d,out=-130,in=30] (6) to (10);
\path [d] (6) to (9);
\path [d,out=-40,in=0] (6) to (12);
\path [d] (7) to (10);
\path [d,out=-20,in=165] (7) to (11);
\path [d] (8) to (12);
\path [d] (8) to (11);
\path [d,out=210,in=10] (9) to (10);
\path [d,out=-95,in=20] (9) to (12);
\path [d] (10) to (12);
\path [d] (11) to (12);
\end{tikzpicture}
$\hspace{-.5cm}$
\raisebox{3.8cm}{\hbox{$\overset\alpha\longrightarrow$}}
$\quad$
\begin{tikzpicture}[scale=1.2]
\node (a) at (6,5) [rect vert,fill=purple!50] {\scriptsize $\multiset{2,3}$};
\node (b) at (4.5,4) [rect vert,fill=orange!50] {\scriptsize $\multiset{2,2}$};
\node (c) at (7.5,4) [rect vert,fill=turq!50] {\scriptsize $\multiset{1,3}$};
\node (d) at (4.5,3) [rect vert,fill=springgreen!50] {\scriptsize $\multiset{1,2}$};
\node (e) at (7.5,3) [rect vert,fill=rose!50] {\scriptsize $\multiset{0,3}$};
\node (f) at (4.5,2) [rect vert,fill=darkorchid!50] {\scriptsize $\multiset{1,1}$};
\node (g) at (7.5,2) [rect vert] {\scriptsize $\multiset{0,2}$};
\node (h) at (6,1) [rect vert,fill=mediterranean!50] {\scriptsize $\multiset{0,1}$};
\node (i) at (6,0) [rect vert,fill=rred!50] {\scriptsize $\multiset{0,0}$};
\node (H) at (6,-.5) {$\mathsf{H}$};

\path [d,out=-100,in=55] (a) to (d);
\path [d] (a) to (b);
\path [d,out=-90] (a) to (g);
\path [d] (a) to (c);
\path [d,out=-80,in=125] (a) to (e);
\path [d] (b) to (d);
\path [d] (b) to (g);
\path [d] (c) to (d);
\path [d] (c) to (f);
\path [d,out=-125,in=90] (c) to (h);
\path [d] (c) to (e);
\path [d] (d) to (f);
\path [d,out=-45,in=120] (d) to (h);
\path [d] (d) to (g);
\path [d,out=-125,in=75] (e) to (h);
\path [d] (e) to (g);
\path [d,out=-30,in=15] (e) to (i);
\path [d] (f) to (h);
\path [d] (g) to (h);
\path [d] (g) to (i);
options h
\path [d] (h) to (i);
\end{tikzpicture}
\caption{
\label{fig:ex of morphism}
A source and option preserving map $\alpha:\mathsf{G}\to\mathsf{H}$ between finite gamegraphs.
}
\end{figure}
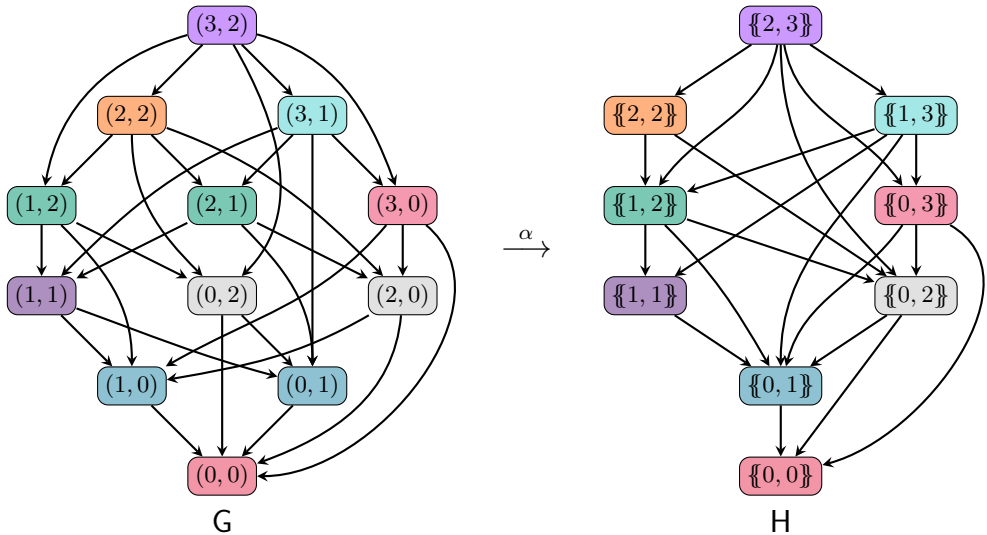

\begin{example}
There are option preserving maps that do not originate from the intuitive notion of symmetry of the positions. One such map is shown in Figure \ref{fig:Wythoff morphism}. This map is between Wythoff's game $\mathsf{G}$ with starting position $\multiset{1,2}$ and the subtraction game $\mathsf{H}$ with $3$ stones in which players are allowed to take at most two stones. Wythoff's game is played on two heaps, where players can remove arbitrarily many stones from one of the heaps or the same number of stones from both. The option and source preserving gamegraph map $\alpha:\mathsf{G}\to\mathsf{H}$ is defined by $\alpha(\multiset{a,b}):=a+b$. Note that this formula  is just a convenient, purely coincidental way to describe $\alpha$ for this particular example, it does not have any deeper meaning and we cannot expect it to work for any larger examples.
\end{example}

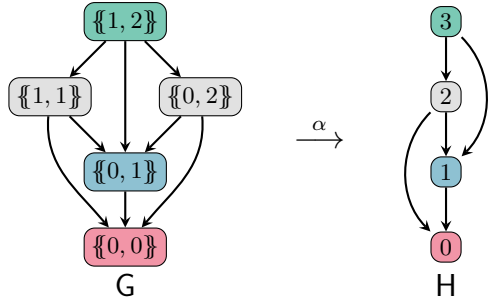
\begin{figure}[ht]
\centering
\begin{tikzpicture}[scale=1,auto]
\node (1) at (0,5) [rect vert,fill=springgreen!50] {\scriptsize $\multiset{1,2}$};
\node (2) at (-1,4) [rect vert] {\scriptsize $\multiset{1,1}$};
\node (3) at (1,4) [rect vert] {\scriptsize $\multiset{0,2}$};
\node (4) at (0,3) [rect vert,fill=mediterranean!50] {\scriptsize $\multiset{0,1}$};
\node (5) at (0,2) [rect vert,fill=rred!50] {\scriptsize $\multiset{0,0}$};
\node (G) at (0,1.5) {$\mathsf{G}$};
\path [d] (1) to (2);
\path [d] (1) to (3);
\path [d] (1) to (4);
\path [d] (2) to (4);
\path [d,out=-95, in=130] (2) to (5);
\path [d] (3) to (4);
\path [d,out=-85, in=50] (3) to (5);
\path [d] (4) to (5);
\end{tikzpicture}
$\quad$
\raisebox{2.1cm}{\hbox{$\overset\alpha\longrightarrow$}}
$\quad$
\begin{tikzpicture}[scale=1]


\node (a) at (6,5) [rect vert,fill=springgreen!50] {\scriptsize $3$};
\node (b) at (6,4) [rect vert] {\scriptsize $2$};
\node (c) at (6,3) [rect vert,fill=mediterranean!50] {\scriptsize $1$};
\node (d) at (6,2) [rect vert,fill=rred!50] {\scriptsize $0$};
\node (H) at (6,1.5) {$\mathsf{H}$};

\path [d] (a) to (b);
\path [d] (b) to (c);
\path [d] (c) to (d);
\path[d,out=-45,in=45] (a) to (c);
 \path[d,out=-135,in=135] (b) to (d);
\end{tikzpicture}

\caption{\label{fig:Wythoff morphism}
A source and option preserving gamegraph map $\alpha:\mathsf{G}\to\mathsf{H}$ that does not originate from a symmetry of the positions.}
\end{figure}

There are examples of games where a human naturally identifies some positions, but further identifications can be made. 

\begin{example}\label{ex:grundy}
In Grundy's game, a move consists of dividing a heap on the table into two non-equal heaps. Figure \ref{fig:Grundy morphism} shows an example starting with a single heap with $7$ stones. Clearly, heaps of size $1$ or $2$ cannot be further divided, so a human player would essentially ignore these heaps as soon as they appear. This results in a natural identification of some positions, which constitutes the option preserving map $\alpha$, shown in the figure. An additional nontrivial identification is possible as shown by the option preserving map $\beta$.

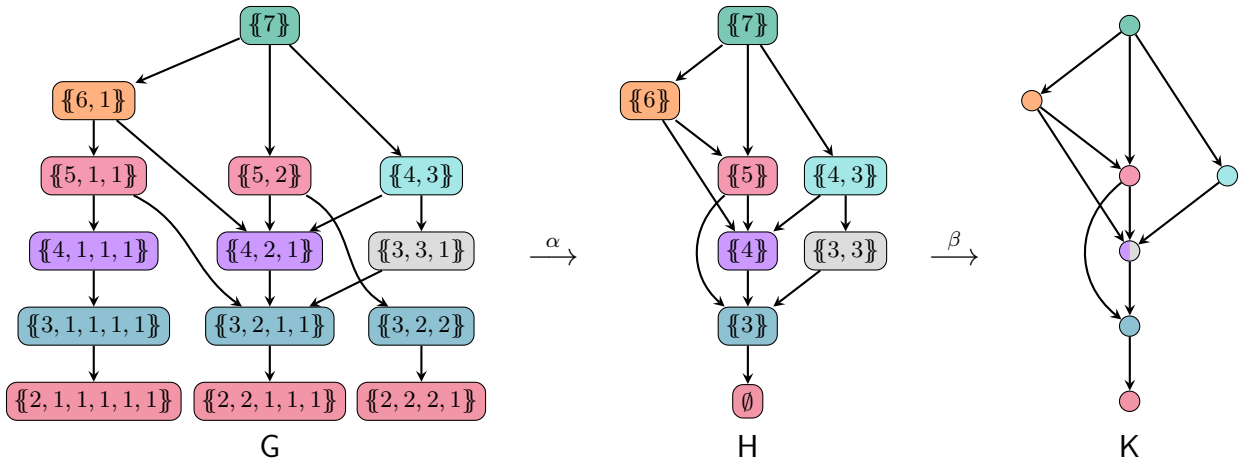
\begin{figure}[ht]
\centering
\begin{tikzpicture}[xscale=1.3,yscale=1]
\node (1) at (0,7) [rect vert,fill=springgreen!50] {\scriptsize $\multiset{7}$};
\node (2) at (-1.8,6) [rect vert,fill=orange!50] {\scriptsize $\multiset{6,1}$};
\node (3) at (-1.8,5) [rect vert,fill=rose!50] {\scriptsize $\multiset{5,1,1}$};
\node (4) at (0,5) [rect vert,fill=rose!50] {\scriptsize $\multiset{5,2}$};
\node (5) at (1.55,5) [rect vert,fill=turq!50] {\scriptsize $\multiset{4,3}$};
\node (6) at (-1.8,4) [rect vert,fill=purple!50] {\scriptsize $\multiset{4,1,1,1}$};
\node (7) at (0,4) [rect vert,fill=purple!50] {\scriptsize $\multiset{4,2,1}$};
\node (8) at (1.55,4) [rect vert,fill=gray] {\scriptsize $\multiset{3,3,1}$};
\node (9) at (-1.8,3) [rect vert,fill=mediterranean!50] {\scriptsize $\multiset{3,1,1,1,1}$};
\node (10) at (0,3) [rect vert,fill=mediterranean!50] {\scriptsize $\multiset{3,2,1,1}$};
\node (11) at (1.55,3) [rect vert,fill=mediterranean!50] {\scriptsize $\multiset{3,2,2}$};
\node (12) at (-1.8,2) [rect vert,fill=rred!50] {\scriptsize $\multiset{2,1,1,1,1,1}$};
\node (13) at (0,2) [rect vert,fill=rred!50] {\scriptsize $\multiset{2,2,1,1,1}$};
\node (14) at (1.55,2) [rect vert,fill=rred!50] {\scriptsize $\multiset{2,2,2,1}$};
\node (G) at (0,1.4) {$\mathsf{G}$};
\path [d] (1) to (2);
\path [d] (1) to (5);
\path [d] (1) to (4);
\path [d] (2) to (3);
\path [d] (2) to (7);
\path [d] (3) to (6);
\path [d,out=-30,in=140] (3) to (10);
\path [d] (4) to (7);
\path [d,out=-32.5,in=145] (4) to (11);
\path [d] (5) to (7);
\path [d] (5) to (8);
\path [d] (6) to (9);
\path [d] (7) to (10);
\path [d] (8) to (10);
\path [d] (9) to (12);
\path [d] (10) to (13);
\path [d] (11) to (14);
\end{tikzpicture}
\quad
\raisebox{2.7cm}{\hbox{$\overset\alpha\longrightarrow$}}
\quad
\begin{tikzpicture}[xscale=1.3,yscale=1]
\node (a) at (4.5,7) [rect vert,fill=springgreen!50] {\scriptsize $\multiset{7}$};
\node (b) at (3.5,6) [rect vert,fill=orange!50] {\scriptsize $\multiset{6}$};
\node (c) at (4.5,5) [rect vert,fill=rose!50] {\scriptsize $\multiset{5}$};
\node (d) at (5.5,5) [rect vert,fill=turq!50] {\scriptsize $\multiset{4,3}$};
\node (e) at (4.5,4) [rect vert,fill=purple!50] {\scriptsize $\multiset{4}$};
\node (f) at (5.5,4) [rect vert,fill=gray] {\scriptsize $\multiset{3,3}$};
\node (g) at (4.5,3) [rect vert,fill=mediterranean!50] {\scriptsize $\multiset{3}$};
\node (h) at (4.5,2) [rect vert,fill=rred!50] {\scriptsize $\emptyset$};
\node (H) at (4.5,1.4) {$\mathsf{H}$};
\path [d] (a) to (b);
\path [d] (a) to (c);
\path [d] (a) to (d);
\path [d] (b) to (c);
\path [d] (b) to (e);
\path [d] (c) to (e);
\path[d,out=-135,in=135] (c) to (g);
\path [d] (d) to (e);
\path [d] (d) to (f);
\path [d] (e) to (g);
\path [d] (f) to (g);
\path [d] (g) to (h);
\end{tikzpicture}
\quad
\raisebox{2.7cm}{\hbox{$\overset\beta\longrightarrow$}}
\quad
\begin{tikzpicture}[xscale=1.3,yscale=1]
\node (a1) at (8.2,7) [small vert, fill=springgreen!50] {};
\node (b1) at (7.2,6) [small vert, fill=orange!50] {};
\node (c1) at (8.2,5) [small vert, fill=rose!50] {};
\node (d1) at (9.2,5) [small vert, fill=turq!50] {};
\node (ee1) at (8.1575,4) [semicircle, rotate=90, fill=purple!50, scale=0.45]{};
\node (ef1) at (8.25,4) [semicircle, rotate=-90, fill=gray, scale=0.45]{};
\node (e1) at (8.2,4) [small vert, fill=none]{};
\node (f1) at (8.2,3) [small vert, fill=mediterranean!50] {};
\node (g1) at (8.2,2) [small vert, fill=rred!50] {};
\node (K) at (8.2,1.4) {$\mathsf{K}$};


\path[d,out=-135,in=135] (c1) to (f1);
\path [d] (a1) to (b1);
\path [d] (a1) to (c1);
\path [d] (a1) to (d1);
\path [d] (b1) to (c1);
\path [d] (b1) to (e1);
\path [d] (c1) to (e1);
\path [d] (d1) to (e1);
\path [d] (e1) to (f1);
\path [d] (f1) to (g1);
\end{tikzpicture}

\caption{
\label{fig:Grundy morphism}
An intermediary, ``natural" option preserving map $\alpha:\mathsf{G}\to\mathsf{H}$,  and option preserving map $\beta:\mathsf{H}\to\mathsf{K}$ that further reduces the gamegraph.}
\end{figure}
\end{example}

\begin{example}\label{ex:47}
\label{example-opm1}
A \emph{geodesic} of a finite graph is a shortest path between two vertices. The \emph{geodetic closure} of a set $S$ of vertices is the set of vertices contained on geodesics between two vertices of $S$. 
In a geodetic achievement game \cite{BeneshGeodetic,BuckleyGeodetic,HaynesGeodetic} on a graph $G$, the players select previously unselected vertices until the geodetic closure of the set $P$ of jointly-selected elements is the full vertex set of $G$. 

Consider the grid graph $G=P_m\boxprod P_n$. It is clear that the geodetic achievement game on $G$ ends as soon as $P$ contains a vertex from each of the four edges of the grid. So to determine the end of the game we only need to know how many vertices are chosen in nine regions of the vertex set. These regions are the four corners, the four edges without corners, and the middle. We map a position to a $3\times 3$ matrix that shows how many unchosen vertices are left in these nine regions. This map $\alpha$ is option preserving between the geodetic achievement game and a matrix game in which players decrease an entry of the $3\times 3$ matrix by $1$. Figure~\ref{fig:geodetic} shows some positions with their images. This matrix game is easier to analyze than the original geodetic achievement game because the option preserving map ignores some  irrelevant information. The nim-numbers of this game were determined in \cite{BeneshGeodetic} using different and somewhat advanced techniques. 
\end{example}

\begin{figure}[ht]
\begin{tikzpicture}[scale=.3,baseline=6pt]
\node at (0,0) {$\circ$};
\node at (0,1) {$\circ$};
\node at (0,2) {$\circ$};
\node at (1,0) {$\circ$};
\node at (1,1) {$\circ$};
\node at (1,2) {$\circ$};
\node at (2,0) {$\circ$};
\node at (2,1) {$\circ$};
\node at (2,2) {$\circ$};
\node at (3,0) {$\circ$};
\node at (3,1) {$\circ$};
\node at (3,2) {$\circ$};
\end{tikzpicture}
$\overset\alpha\longmapsto$
$\left[
\begin{smallmatrix}
1 & 2 & 1 \\
1 & 2 & 1 \\
1 & 2 & 1
\end{smallmatrix}
\right]$
\hfil$\cdots$\hfil
\begin{tikzpicture}[scale=.3,baseline=6pt]
\node at (0,0) {$\circ$};
\node at (0,1) {$\bullet$};
\node at (0,2) {$\circ$};
\node at (1,0) {$\circ$};
\node at (1,1) {$\bullet$};
\node at (1,2) {$\circ$};
\node at (2,0) {$\circ$};
\node at (2,1) {$\circ$};
\node at (2,2) {$\bullet$};
\node at (3,0) {$\circ$};
\node at (3,1) {$\circ$};
\node at (3,2) {$\circ$};
\end{tikzpicture}
$\overset\alpha\longmapsto$
$\left[
\begin{smallmatrix}
1 & 1 & 1 \\
0 & 1 & 1 \\
1 & 2 & 1
\end{smallmatrix}
\right]$
\hfil$\cdots$\hfil
\begin{tikzpicture}[scale=.3,baseline=6pt]
\node at (0,0) {$\circ$};
\node at (0,1) {$\bullet$};
\node at (0,2) {$\circ$};
\node at (1,0) {$\bullet$};
\node at (1,1) {$\bullet$};
\node at (1,2) {$\circ$};
\node at (2,0) {$\bullet$};
\node at (2,1) {$\circ$};
\node at (2,2) {$\bullet$};
\node at (3,0) {$\bullet$};
\node at (3,1) {$\circ$};
\node at (3,2) {$\circ$};
\end{tikzpicture}
$\overset\alpha\longmapsto$
$\left[
\begin{smallmatrix}
1 & 1 & 1 \\
0 & 1 & 1 \\
1 & 0 & 0
\end{smallmatrix}
\right]$

\caption{
\label{fig:geodetic}
Start, middle, and end stages of the geodetic achievement game on $P_3\boxprod P_4$. Both the original positions and their images through $\alpha$ are shown. Filled circles indicate the chosen vertices. Edges of the graph are not shown.}
\end{figure}

\begin{remark}
If $\alpha:\mathsf{S}\to\mathsf{T}$ and $\beta:\mathsf{R}\to\mathsf{S}$ are option preserving rulegraph maps, then
\[
\alpha(\beta(\Opt(p)))=\alpha(\Opt(\beta(p)))=\Opt(\alpha(\beta(p))).
\]
for all $p$. So the composition $\alpha\circ\beta$ is also option preserving. Similarly, the composition of source preserving maps is also source preserving.
\end{remark}

\begin{proposition}
\label{prop:sourceIffSurj}
Let $\alpha: \mathsf{G}\to \mathsf{H}$ be an option preserving gamegraph map. Then $\alpha$ is source preserving if and only if $\alpha$ is surjective.
\end{proposition}

\begin{proof}
First suppose that $\alpha$ is source preserving. If position $r$ is in the range of $\alpha$ and $s\in\Opt(r)$, then $s$ is also in the range of $\alpha$ since $r=\alpha(p)$ for some $p$ and so $s\in\Opt(\alpha(p))=\alpha(\Opt(p))$. This implies that every subposition of $r$ is also in the range of $\alpha$. Since the starting position of $\mathsf{H}$ is in the range of $\alpha$, every position of $\mathsf{H}$ is in the range of $\alpha$.

Now suppose that $\alpha$ is surjective. Then the starting position $r_0$ of $\mathsf{H}$ is in the range of $\alpha$. That is, $r_0=\alpha(q)$ for some $q$. For a contradiction suppose that $q$ is not the starting position of $\mathsf{G}$ and so $q\in\Opt(p)$ for some $p$. Then $r_0=\alpha(q)\in\alpha(\Opt(p))=\Opt(\alpha(p))$, which is a contradiction.
\end{proof}

\begin{example}\label{ex:option preserving not surjective}
The previous proposition does not generalize to rulegraphs. The rulegraph maps shown in Figure~\ref{fig:InfiniteCounterexamples} are option preserving and take sources to sources injectively but they are not surjective. The first fails to be surjective because not every source of the codomain is in the range. The image of the second one contains every source of the codomain but not every position of the codomain is a subposition of some source.
\end{example}

\begin{figure}[ht]
\centering
 \begin{tikzpicture}[scale=.8]
 \node (t) at (9.7,1) {$\dots$};
 \node (2) [small vert,fill=purple!50,star,star points=5] at (8,1) {};
 \node (3) [small vert,fill=rose!50,star,star points=6] at (7,1) {};
 \node (4) [small vert,fill=turq!50,star,star points=4] at (9,1) {};
 \node (5) [small vert,fill=springgreen!50,star,star points=3] at (8,0) {};
 \path [d] (3) to (5);
 \path [d] (4) to (5);
 \path [d] (2) to (5);
\end{tikzpicture}
\hspace{0cm}
\raisebox{.8cm}{\hbox{$\longrightarrow$}}
$\quad$
\begin{tikzpicture}[scale=.8]
 \node (t) at (16.7,1) {$\dots$};
 \node (e) [small vert,fill=purple!50,star,star points=5] at (15,1) {};
 \node (f) [small vert,fill=rose!50,star,star points=6] at (14,1) {};
 \node (g) [small vert,fill=turq!50,star,star points=4] at (16,1) {};
 \node (h) [small vert,fill=springgreen!50,star,star points=3] at (14,0) {}; 
 \node (i) [small vert] at (13,1) {};
 \path [d] (e) to (h);
 \path [d] (i) to (h);
 \path [d] (f) to (h);
 \path [d] (g) to (h);
 \end{tikzpicture}
 \qquad\qquad
 \begin{tikzpicture}[scale=.8]
 \node (1) [small vert,fill=orange!50,star,star points=4] at (-0.25,1) {};
 \node (0) [small vert,fill=mediterranean!50,star,star points=3] at (-0.25,0) {};
 \path [d] (1) to (0);
 \end{tikzpicture}
\quad
\raisebox{.8cm}{\hbox{$\longrightarrow$}}
$\quad$
\begin{tikzpicture}[scale=.8]
 \node (a) at (4,2.75) {$\vdots$};
 \node (b) [small vert,fill=orange!50,star,star points=4] at (3,1) {}; 
 \node (c) [small vert] at (4,1) {};
 \node (e) [small vert] at (4,2) {};
 \node (d) [small vert,fill=mediterranean!50,star,star points=3] at (3.5,0) {};
 \path [d] (e) to (c);
 \path [d] (b) to (d);
 \path [d] (c) to (d);
\end{tikzpicture}

\caption{
\label{fig:InfiniteCounterexamples}
Option preserving maps on rulegraphs that take sources to sources but neither is surjective.}
\end{figure}
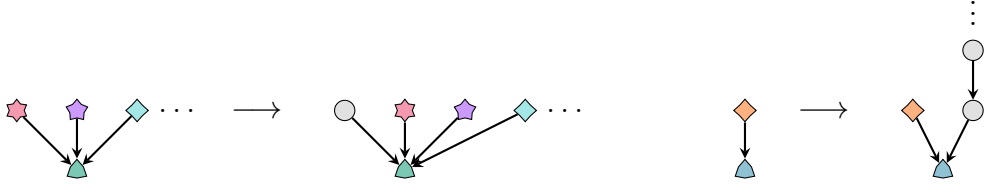

\begin{definition}
A map $f:V(\mathsf{R})\to Y$ on the positions of a rulegraph $\mathsf{R}$ is called a \emph{valuation} if there is a map $\mu:2^Y\to Y$ such that $f(p)=\mu(f(\Opt(p)))$ for each position $p$.
\end{definition}

Note that fixing a $\mu$ determines a corresponding $f$ that is defined for any rulegraph. So we are going to denote this valuation by $f$ even if we are working with several rulegraphs.

If $f$ is a valuation on a gamegraph $\mathsf{G}$ with starting position $s$, then we write $f(\mathsf{G}):=f(s)$.

\begin{example}
The nim-number of positions is a valuation with $\mu=\mex$.
\end{example}

\begin{example}
The normal play outcome function $o^+:V(\mathsf{R})\to\{P,N\}$ is a valuation with
\[
\mu^+(A)=\begin{cases}
N, & P\in A; \\
P, & \text{otherwise}.\\
\end{cases}
\]
\end{example}

\begin{example}
The mis\`{e}re play outcome function $o^-:V(\mathsf{R})\to\{P,N\}$ is a valuation with
\[
\mu^-(A)=\begin{cases}
N, & P\in A \text{ or } A=\emptyset;\\
P, & \text{otherwise}.\\
\end{cases}
\]
\end{example}

\begin{example}
The formal birthday $\fbd(p)$ of a position $p$ is a valuation with 
\[
\mu(A)=\sup\{x+1 \mid x\in A \}.
\]
\end{example}

\begin{example}
The minimum distance of a position from a terminal position is a valuation with 
\[
\mu(A)=\begin{cases}
0, & A=\emptyset;\\
\min\{x+1 \mid x\in A \}, & \text{otherwise}.\\
\end{cases}
\]
\end{example}

The following result is the main reason for our interest in option preserving maps.

\begin{proposition}\label{prop:nim preserving}
If $\alpha: \mathsf{R}\to\mathsf{S}$ is an option preserving rulegraph map and $f$ is a valuation, then $f(\alpha(p))=f(p)$ for each position $p$ in $\mathsf{R}$.
\end{proposition}

\begin{proof}
Using transfinite structural induction on the positions, we have
\[
\begin{aligned}
f(\alpha(p))
 &=\mu(f(\Opt(\alpha(p)))) \\
 &=\mu(f(\alpha(\Opt(p)))) & & \text{($\alpha$ is option preserving)} \\
 &=\mu(f(\Opt(p))) & & \text{(by induction)} \\
 &=f(p).
\end{aligned}
\]
\end{proof}

\begin{corollary}
If $\alpha:\mathsf{G}\to\mathsf{H}$ is a source and option preserving gamegraph map, then $\nim(\mathsf{G})=\nim(\mathsf{H})$, $o^+(\mathsf{G})=o^+(\mathsf{H})$,  $o^-(\mathsf{G})=o^-(\mathsf{H})$, and $\fbd(\mathsf{G})=\fbd(\mathsf{H})$.
\end{corollary}

\begin{example}
\label{ex:mouse}
Consider a rectangular maze in which a mouse starts from the lower left corner and can move up or right as shown in Figure~\ref{fig:maze}. Two players take turns moving the mouse in the maze. The game ends as soon as the mouse reaches the top or the right edge of the maze. A piece of cheese is placed at the top edge while a cat is waiting for the mouse on the right edge. So the top edge contains $N$-positions while the right edge contains $P$-positions. Figure~\ref{fig:maze} shows the outcome of each position in a $3\times 4$ maze. The outcome at a nonterminal position $p$ is determined by the usual $o(p):=\mu^+(o(\Opt(p)))$. 

This outcome function $o$ is not a valuation since $\Opt(p)=\emptyset$ at each terminal position $p$ but $o(p)$ can be either $N$ or $P$, so there is no compatible way to define $\mu(\emptyset)$. As a consequence, we cannot expect that option preserving maps preserve the outcome of the positions. One might want to develop a theory of a category whose objects are rulegraphs with an outcome defined on their terminal positions and whose morphisms are option preserving maps that also preserve the outcomes at terminal positions. This mixed normal and mis\`{e}re play cannot be studied using the traditional set of options model of games.  
\end{example}

\begin{figure}[ht]
\begin{tikzpicture}[scale=.9]
\draw (-.5,-.5) -- (-.5,2.5) -- (3.5,2.5) -- (3.5,-0.5) -- cycle;
\node (00)  at (0,0) {$P$};
\node (10)  at (1,0) {$N$};
\node (20)  at (2,0) {$N$};
\node (30) at (3,0) {$P$};
\node (01)  at (0,1) {$N$};
\node (11)  at (1,1) {$P$};
\node (21)  at (2,1) {$N$};
\node (31) at (3,1) {$P$};
\node (02)  at (0,2) {$N$};
\node (12)  at (1,2) {$N$};
\node (22)  at (2,2) {$N$};
\node at (4,.5) {cat};
\node at (1,3) {cheese};
\path [d] (00) to (10);
\path [d] (10) to (20);
\path [d] (20) to (30);
\path [d] (01) to (11);
\path [d] (11) to (21);
\path [d] (21) to (31);
\path [d] (00) to (01);
\path [d] (01) to (02);
\path [d] (10) to (11);
\path [d] (11) to (12);
\path [d] (20) to (21);
\path [d] (21) to (22);
\end{tikzpicture}

\caption{
\label{fig:maze}
The outcome of the positions in the mouse in the maze game.
}
\end{figure}
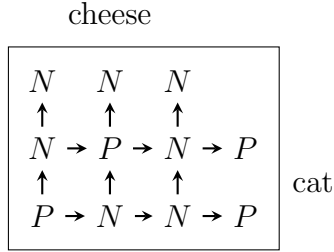

Roughly speaking, the next result states that an option preserving rulegraph map $\alpha:\mathsf{R}\to \mathsf{S}$ takes arrows to arrows and every arrow connecting positions in $\alpha(V(\mathsf{R}))$ is the image of an arrow. In other words, every option preserving rulegraph map is a faithful digraph homomorphism, where the notion of faithful is the digraph version of the definition in~\cite{HahnTardiff}.

\begin{proposition}
\label{prop:option preserving implies arrows preserved}
If $\alpha: C\to D$ is an option preserving digraph map, then $q\in\Opt(p)$ in $C$ implies $\alpha(q)\in\Opt(\alpha(p))$ in $D$. Moreover, if $\alpha(q)\in\Opt(\alpha(p))$, then $\alpha(q)=\alpha(r)$ for some $r$ satisfying $r\in\Opt(p)$.
\end{proposition}

\begin{proof}
Suppose $q\in\Opt(p)$ in $C$. Since $\alpha$ is option preserving, $\alpha(q)\in\alpha(\Opt(p))=\Opt(\alpha(p))$, so $\alpha(q)\in\Opt(\alpha(p))$ in $D$.

Now, suppose $\alpha(q)\in\Opt(\alpha(p))$. Hence $\alpha(q)\in\alpha(\Opt(p))$, and so there exists $r\in\Opt(p)$ such that $\alpha(q)=\alpha(r)$.
\end{proof}

The converse of the previous result is not true.

\begin{example}
Figure~\ref{fig:notOptPreserving} shows a rulegraph map $\alpha:\mathsf{R}\to\mathsf{S}$ that is not option preserving but satisfies both conclusions of Proposition~\ref{prop:option preserving implies arrows preserved}. Note that $\alpha$ is a faithful digraph homomorphism but it does not preserve outcomes.
\end{example}

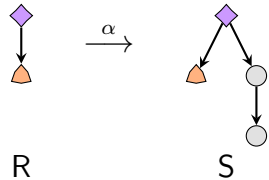
\begin{figure}[ht]
\begin{tikzpicture}[scale=.8]
\node (02) [small vert,fill=purple!50,star,star points=4] at (0,2) {};
\node (01) [small vert,fill=orange!50,star,star points=3] at (0,1) {};
\path [d] (02) to (01);
\node (R) at (0,-.5) {$\mathsf{R}$};
\end{tikzpicture}
\quad
\raisebox{1.8cm}{\hbox{$\overset\alpha\longrightarrow$}}
$\quad$
\begin{tikzpicture}[scale=.8]
\node (a) [small vert,fill=purple!50,star,star points=4] at (3,2) {};
\node (b) [small vert,fill=orange!50,star,star points=3] at (2.5,1) {};
\node (c) [small vert] at (3.5,1) {};
\node (d) [small vert] at (3.5,0) {};
\path [d] (a) to (b);
\path [d] (a) to (c);
\path [d] (c) to (d);
\node (S) at (3,-.5) {$\mathsf{S}$};
\end{tikzpicture}

\caption{
\label{fig:notOptPreserving}
A rulegraph map that is not option preserving.}
\end{figure}

The next result shows that if one is using an option preserving map to analyze a rulegraph, then the image is automatically a rulegraph.

\begin{proposition}
\label{prop:image ruleset}
If $\alpha:\mathsf{R}\to D$ is an option preserving map from a rulegraph to a digraph, then the digraph induced by $\alpha(V(\mathsf{R}))$ in $D$ is a rulegraph.   Moreover if $\mathsf{R}$ is a gamegraph, then the digraph induced by $\alpha(V(\mathsf{R}))$ is also a gamegraph. 
\end{proposition}

\begin{proof}
For a contradiction suppose that $\alpha(p_1),\alpha(p_2),\alpha(p_3),\ldots$ is an infinite walk in $D$. Using the second part of Proposition~\ref{prop:option preserving implies arrows preserved}, we can recursively find a walk $p_1',p_2',p_3',\ldots$ such that $\alpha(p_i')=\alpha(p_i)$ for all $i\ge 2$. This is a contradiction since $\mathsf{R}$ does not have any infinite walk.

Furthermore, if $\mathsf{R}$ is a gamegraph, then every position $\alpha(p)$ in $\alpha(V(\mathsf{R}))$ is a subposition of the image $\alpha(s)$ of the source $s$ of $\mathsf{R}$ since the image of a path from $s$ to $p$ is a path from $\alpha(s)$ to $\alpha(p)$.
\end{proof}

In light of Proposition~\ref{prop:image ruleset}, we define the \emph{image rulegraph} $\alpha(\mathsf{R})$ to be the rulegraph induced by $\alpha(V(\mathsf{R}))$ in $D$.

\section{Categories}

The class of rulegraphs with option preserving maps forms a concrete category $\mathbf{RGph}$. Proposition~\ref{prop:option preserving implies arrows preserved} implies that $\mathbf{RGph}$ is a subcategory of $\mathbf{Gph}$. Note that this subcategory is neither wide nor full. The category $\mathbf{GGph}$ of gamegraphs and option preserving maps is a full but not wide subcategory of $\mathbf{RGph}$. The category $\mathbf{GGph_s}$ of gamegraphs and option and source preserving maps is a wide subcategory of $\mathbf{GGph}$, but is not a full subcategory.

Outcomes form a category $\mathbf{Out}$ with objects $P$ and $N$ and only identity morphisms. The map that takes gamegraphs to their outcomes is a functor from $\mathbf{GGph_s}$ to $\mathbf{Out}$. This is true for both normal and mis\`{e}re outcomes.

Recall that an isomorphism in a category is an invertible morphism.

\begin{example}
Consider the digraph map $\alpha:V(\vec K_2^\complement)\to V(\vec K_2)$ shown in Figure~\ref{fig:bijective hom}. It is clear that $\alpha$ is a bijective homomorphism but $\alpha^{-1}$ is not a homomorphism. This shows that in $\mathbf{Gph}$ a bijective homomorphism might not be an isomorphism. In fact, the isomorphisms in this category are bijective maps for which $(p,q)$ is an arrow if and only if $(\alpha(p),\alpha(q))$ is an arrow.  
\end{example}

\begin{figure}[ht]
\begin{tikzpicture}[scale=.8]
\node (02) [small vert,fill=orange!50,star,star points=4] at (-1,1) {};
\node (01) [small vert,fill=springgreen!50,star,star points=3] at (0,1) {};
\end{tikzpicture}
\quad
\raisebox{.07cm}{\hbox{$\overset\alpha\longrightarrow$}}
$\quad$
\begin{tikzpicture}[scale=.8]
\node (b) [small vert,fill=orange!50,star,star points=4] at (3,1) {};
\node (d) [small vert,fill=springgreen!50,star,star points=3] at (4,1) {};
\path [d] (b) to (d);
\end{tikzpicture}

\caption{
\label{fig:bijective hom}
A bijective digraph homomorphism that is not an isomorphism.}
\end{figure}
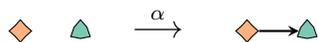

The situation is less complicated in $\mathbf{RGph}$.

\begin{proposition}\label{prop:inverse of bijective option preserving}
The inverse of a bijective option preserving rulegraph map $\alpha: \mathsf{R}\to \mathsf{S}$ is also option preserving.
\end{proposition}

\begin{proof}
Let $q$ be a position of $\mathsf{S}$. Then $q=\alpha(p)$ for some $p$. Applying $\alpha^{-1}$ to the identity $\Opt(\alpha(p))=\alpha(\Opt(p))$ gives $\alpha^{-1}(\Opt(q))=\Opt(\alpha^{-1}(q))$.   
\end{proof}

\begin{proposition}
\label{prop:categoryIso}
The isomorphisms in $\mathbf{RGph}$, $\mathbf{GGph}$, and $\mathbf{GGph_s}$ are the bijective morphisms.
\end{proposition}

\begin{proof}
In all of these categories, an isomorphism must be a bijective morphism. On the other hand, the inverse of a bijective option preserving map is option preserving by Proposition~\ref{prop:inverse of bijective option preserving}. This shows that in $\mathbf{RGph}$ and $\mathbf{GGph}$ the bijective morphisms are isomorphisms. The inverse of a bijective morphism in $\mathbf{RGph}_s$ is an option preserving surjective map and hence source preserving by Proposition~\ref{prop:sourceIffSurj}. Thus the bijective morphisms in $\mathbf{GGph_s}$ are isomorphisms.
\end{proof}





The next result says that $\mathbf{RGph}$, $\mathbf{GGph}$, and $\mathbf{GGph_s}$ are isomorphism-closed subcategories of $\mathbf{Gph}$.

\begin{proposition}
The isomorphisms in $\mathbf{RGph}$, $\mathbf{GGph}$, and $\mathbf{GGph_s}$ are the digraph isomorphisms.
\end{proposition}

\begin{proof}
It is clear that a digraph isomorphism between appropriate objects is also an isomorphism in each of these categories. On the other hand, an isomorphism in one of these subcategories of $\mathbf{Gph}$ must be isomorphism in $\mathbf{Gph}$.
\end{proof}


\begin{example}
Figure~\ref{fig:SB} shows two non-isomorphic rulegraphs $\mathsf{R}$ and $\mathsf{S}$. One can find injective option preserving maps $\mathsf{R}\to\mathsf{S}$ and $\mathsf{S}\to\mathsf{R}$. The range of the $\mathsf{R}\to\mathsf{S}$ map is the vertex set of $\mathsf{S}$ without the leftmost vertex, while the range of the $\mathsf{S}\to\mathsf{R}$ map is the vertex set of $\mathsf{R}$ without the top-left vertex. Thus $\mathbf{RGph}$ does not have the Schr\"oder--Bernstein property.
\end{example}

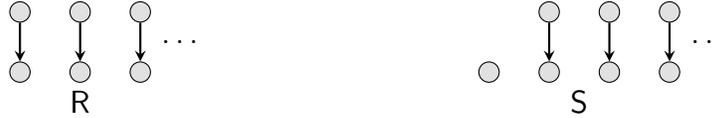
\begin{figure}[ht]
\centering
\begin{tikzpicture}[scale=.8]
\node (t) at (9.7,.5) {$\dots$};
\node (11) [small vert] at (7,1) {};
\node (22) [small vert] at (8,1) {};
\node (33) [small vert] at (9,1) {};
\node (1) [small vert] at (7,0) {};
\node (2) [small vert] at (8,0) {};
\node (3) [small vert] at (9,0) {};
\path [d] (11) to (1);
\path [d] (22) to (2);
\path [d] (33) to (3);
\node (R) at (8,-.5) {$\mathsf{R}$};
\end{tikzpicture}
\hfil
\begin{tikzpicture}[scale=.8]
\node (t) at (9.7,.5) {$\dots$};
\node (11) [small vert] at (7,1) {};
\node (22) [small vert] at (8,1) {};
\node (33) [small vert] at (9,1) {};
\node (1) [small vert] at (7,0) {};
\node (2) [small vert] at (8,0) {};
\node (3) [small vert] at (9,0) {};
\node     [small vert] at (6,0) {};
\path [d] (11) to (1);
\path [d] (22) to (2);
\path [d] (33) to (3);
\node (S) at (7.5,-.5) {$\mathsf{S}$};
\end{tikzpicture}

\caption{
\label{fig:SB}
Two non-isomorphic rulegraphs.
}
\end{figure}

One consequence of the next result is that $\mathbf{GGph}$ has the Schr\"oder--Bernstein property.

\begin{proposition}
If $\alpha:\mathsf{G}\to \mathsf{H}$ and $\beta:\mathsf{H}\to \mathsf{G}$ are option preserving gamegraph maps, then both $\alpha$ and $\beta$ are source preserving.
\end{proposition}

\begin{proof}
For a contradiction suppose without loss of generality that $\alpha$ is not source preserving. Then the source $s_{\mathsf{G}}$ of $\mathsf{G}$ maps to a position $\alpha(s_{\mathsf{G}})$ that is different from the source $s_{\mathsf{H}}$ of $\mathsf{H}$. Since $\mathsf{H}$ is a gamegraph, $\alpha(s_{\mathsf{G}})$ is a subposition of $s_{\mathsf{H}}$. Hence $\fbd(\alpha(s_{\mathsf{G}}))<\fbd(s_{\mathsf{H}})$. Since $\beta$ preserves formal birthdays and $\beta(s_{\mathsf{H}})$ is a subposition of $s_{\mathsf{G}}$, 
\[
\fbd(\beta(\alpha(s_{\mathsf{G}})))<\fbd(\beta(s_{\mathsf{G}}))\le \fbd(s_{\mathsf{G}}).
\]
This is a contradiction since $\beta\circ\alpha$ also preserves formal birthdays.
\end{proof}

The maps $\alpha$ and $\beta$ in the previous proposition must be surjective by Proposition~\ref{prop:sourceIffSurj}. Thus, if $\alpha$ or $\beta$ are also injective, then both are also isomorphisms.

\section{Quotients}\label{sec:quotients}

Recall that a congruence relation is an equivalence relation that is compatible with the structure. 

\begin{definition}
An equivalence relation $\sim$ on the positions of a rulegraph $\mathsf{R}$ is a \emph{congruence relation} if $p\sim q$ implies $[\Opt(p)]=[\Opt(q)]$ with the interpretation $[S]=\{[s]\mid s\in S\}$.
\end{definition}

\begin{example}
The diagonal equivalence relation $\Delta$, whose equivalence classes are singleton sets, is a congruence relation on any rulegraph. The trivial equivalence relation $\nabla$ with only one equivalence class is not a congruence relation on a rulegraph unless every position is terminal. 
\end{example}

In the context of gamegraphs, a partition induced by a congruence relation is called congruential in \cite{EmulationalEquivalence}.

\begin{proposition}
\label{prop:kernel is congruence relation}
If $\alpha: \mathsf{R}\to \mathsf{S}$ is an option preserving rulegraph map, then the kernel of $\alpha$ is a congruence relation.
\end{proposition}

\begin{proof}
Let $p\sim q$, where $\sim$ is the kernel of $\alpha$. That is, $\alpha(p)=\alpha(q)$. We show that $[\Opt(p)]\subseteq[\Opt(q)]$. The reverse containment will follow by symmetry. Assume $[r]\in [\Opt(p)]$, so that $r\sim s$ for some $s\in\Opt(p)$. Then 
\[
\alpha(r)=\alpha(s)\in\alpha(\Opt(p))=\Opt(\alpha(p))=\Opt(\alpha(q))=\alpha(\Opt(q)).
\]
Hence $\alpha(r)=\alpha(t)$ for some $t\in\Opt(q)$. Thus $r\sim t$ and so $[r]\in [\Opt(q)]$.
\end{proof}

Let $\sim$ be a congruence relation on the positions of rulegraph $\mathsf{R}$. The quotient digraph $\mathsf{R}/{\sim}$ has vertex set $V(\mathsf{R})/{\sim}$ and arrow set $\{([p],[q])\mid p\in V(\mathsf{R}),\ q\in\Opt(p)\}$. 

\begin{lemma}
\label{lem:compatible}
Let $\sim$ be a congruence relation on a rulegraph. If $[q]\in\Opt([p])$, then there exists an $r\in[q]\cap \Opt(p)$.
\end{lemma}

\begin{proof}
Since $[q]\in\Opt([p])$, $q'\in\Opt(p')$ for some $p'\in[p]$ and $q'\in[q]$. Since $p\sim p'$ and $\sim$ is a congruence relation, 
\[
\{[r]\mid r\in \Opt(p)\}=[\Opt(p)]=[\Opt(p')]=\{[r']\mid r'\in \Opt(p')\}.
\]
Since $[q']\in[\Opt(p')]$, $[q']=[r]$ for some $r\in \Opt(p)$. This $r$ satisfies $r\in[q']=[q]$.
\end{proof}

\begin{proposition}\label{prop:quotient is ruleset}
If $\mathsf{R}$ is a rulegraph and $\sim$ is a congruence relation, then $\mathsf{R}/{\sim}$ is a rulegraph.
\end{proposition}

\begin{proof}
For a contradiction suppose that $[p_1],[p_2],[p_3],\ldots$ is an infinite walk in $\mathsf{R}/{\sim}$. We will recursively find an infinite walk $p_1,p_2^*,p_3^*,\ldots$ in $\mathsf{R}$ such that $p_i^*\sim p_i$ for all $i\ge 2$. Applying Lemma~\ref{lem:compatible} to $[p_2]\in\Opt([p_1])$ gives a $p^*_2\in [p_2]\cap \Opt(p_1)$.
Now, applying Lemma~\ref{lem:compatible} to $[p_3]\in\Opt([p_2])=\Opt([p_2^*])$ gives a $p^*_3\in [p_3]\cap \Opt(p_2^*)$. Iterating this process creates the infinite walk in $\mathsf{R}$, which is a contradiction.
\end{proof}





Note that as a consequence the quotient digraph has no loops. In light of the previous proposition, we refer to $\mathsf{R}/{\sim}$ as a \emph{quotient rulegraph}. Note that in the quotient rulegraph $\Opt([p])=\{[q]\mid q\in\Opt(p)\}$.

\begin{remark}
\label{rem:mustEqual}
Assume $\Opt([p])$ does not contain any nontrivial congruence classes. If $p\sim q$, then 
\[
\{\{r\}\mid r\in\Opt(p)\}=[\Opt(p)]=
[\Opt(q)]=\{\{r\}\mid r\in\Opt(q)\},
\]
and so $\Opt(p)=\Opt(q)$. In particular, if a congruence relation only identifies the positions in a set $A$, then $p,q\in A$ implies $\Opt(p)=\Opt(q)$.
\end{remark}

Loosely speaking, the next result tells us that there is no ``vertical collapsing" in a quotient rulegraph.

\begin{proposition}
\label{prop:noVerticalCollapse}
If $\sim$ is a congruence relation on a rulegraph $\mathsf{R}$ and $q$ is a subposition of $p$, then $p$ and $q$ are not related unless $p=q$. In particular, if $\mathsf{R}$ is a gamegraph, then the congruence class of the starting position $s$ is $[s]=\{s\}$.
\end{proposition}

\begin{proof}
For a contradiction suppose $p\sim q$ and there is a walk ${p=p_0},p_1, \ldots, {p_n=q}$.  Then $[p_0], [p_1], \ldots,  {[p_n]=[p_0]}$ is a cycle in $\mathsf{R}/{\sim}$, contradicting~Proposition \ref{prop:quotient is ruleset}.
\end{proof}

\begin{proposition}
If $\mathsf{G}$ is a gamegraph with starting position $s$ and $\sim$ is a congruence relation, then $\mathsf{G}/{\sim}$ is a gamegraph whose starting position is $[s]$.
\end{proposition}

\begin{proof}
Proposition~\ref{prop:quotient is ruleset} implies that $\mathsf{G}/{\sim}$ has no infinite walks. We show that $[s]$ is a source of $\mathsf{G}/{\sim}$. For a contradiction suppose that $[s]\in\Opt([p])$ for some $p$. Then $s\sim q$ for some $q\in\Opt(p)$. This is a contradiction since $s=q$ by Proposition~\ref{prop:noVerticalCollapse}. Now we show that $[s]$ is the only source of $\mathsf{G}/{\sim}$. For a contradiction suppose that $[r]$ is also a source. Since $r$ is a subposition of $s$, $r\in\Opt(p)$ for some $p$. Hence $[r]\in\Opt([p])$, which is a contradiction. Proposition~\ref{prop:noVerticalCollapse} together with the definition of the quotient rulegraph implies that every position of $\mathsf{G}/{\sim}$ is a subposition of $[s]$.
\end{proof}

We refer to $\mathsf{G}/{\sim}$ as a \emph{quotient gamegraph}.

\begin{proposition}
\label{prop:quotientPreserves}
If $\sim$ is a congruence relation on the rulegraph $\mathsf{R}$, then the quotient map $p\mapsto[p]:\mathsf{R}\to \mathsf{R}/{\sim}$ is option preserving.
\end{proposition}

\begin{proof}
The result follows from $\Opt([p])=\{[q]\mid q\in\Opt(p)\}=[\Opt(p)]$.
\end{proof}

\begin{proposition}
\label{prop:CongPreservesValuation}
If $\sim$ is a congruence relation on a rulegraph and $f$ is a valuation, then $f(p)=f(q)$ for all $p\sim q$.
\end{proposition}

\begin{proof}
The result follows from Propositions~\ref{prop:nim preserving} and \ref{prop:quotientPreserves}.
\end{proof}

The following result can be thought of as the First Isomorphism Theorem for rulegraphs.

\begin{proposition}
\label{prop:fundamental homomorphism theorem}
If $\alpha:\mathsf{R}\to \mathsf{S}$ is an option preserving rulegraph map, then $\alpha(\mathsf{R})$ is isomorphic to the quotient rulegraph $\mathsf{R}/\alpha$.
\end{proposition}

\begin{proof}
Using Proposition~\ref{prop:categoryIso}, it is easy to see that $[p]\mapsto\alpha(p):\mathsf{R}/\alpha\to\alpha(\mathsf{R})$ is a rulegraph isomorphism.
\end{proof}

\begin{proposition}
If $\alpha:\mathsf{G}\to \mathsf{H}$ is a source and option preserving gamegraph map, then $\mathsf{G}/\alpha$ is isomorphic to $\mathsf{H}$.
\end{proposition}

\begin{proof}
The result follows from Propositions~\ref{prop:fundamental homomorphism theorem} and ~\ref{prop:sourceIffSurj}.
\end{proof}

\section{Minimum quotient}\label{sec:simplification}

\begin{definition}
We call a rulegraph $\mathsf{R}$ \emph{simple} if every congruence relation on $V(\mathsf{R})$ is trivial. 
\end{definition}

The notion of simple is analogous to the same term in the context of universal algebras. Indeed, many of the results in this paper have direct analogs in the algebraic setting. 

The next result immediately follows from Propositon~\ref{prop:fundamental homomorphism theorem}.

\begin{proposition}
\label{prop:slim iff injective}
A rulegraph $\mathsf{R}$ is simple if and only if every option preserving map $\alpha:\mathsf{R}\to\mathsf{S}$ is injective.
\end{proposition}

The following proposition provides an easy to check characterization of simple rulegraphs. 

\begin{proposition}
\label{prop:injective}
A rulegraph is simple if and only if the $\Opt$ map is injective.
\end{proposition} 

\begin{proof}
For both implications, we will prove the contrapositive. 

First, suppose two different positions have the same options. Then we can identify these two positions, and hence create a nontrivial congruence relation.

Now, assume there is a nontrivial congruence relation. Pick a nontrivial congruence class $P$ that has no nontrivial congruence class as an option. Such $P$ exists, since the quotient rulegraph has no infinite walks. Any two positions in $P$ have the same option set, so the $\Opt$ map is not injective.
\end{proof}

A digraph in which no two vertices have the same out-neighborhood is called \emph{extensional} in \cite{DigraphParameters,extensionalAcyclic}.

\begin{proposition}
\label{prop:uniqueSlim}
If the quotients $\mathsf{R}/{\sim}$ and $\mathsf{R}/{\approx}$ of the rulegraph $\mathsf{R}$ are both simple, then the congruence relations $\sim$ and $\approx$ are the same.  
\end{proposition}

\begin{proof}
For a contradiction, suppose the two relations $\sim$ and $\approx$ are not the same. We let $[p]$ denote the congruence class of $p$ with respect to $\sim$. Then without loss of generality we can find positions $p$ and $q$ with smallest formal birthday such that $p\approx q$ but $p\not\sim q$. Note that the formal birthdays of $p$ and $q$ are the same by Proposition~\ref{prop:CongPreservesValuation}. We show that $\mathsf{R}/{\sim}$ is not simple because $\Opt([p])=\Opt([q])$. Assume $x\in\Opt(p)$, so that $[x]\in\Opt([p])$. Since $p\approx q$, there is a $y\in\Opt(q)$ such that $x\approx y$. The choice of $p$ and $q$ implies that $x\sim y$. Hence $[x]=[y]\in\Opt([q])$. Thus $\Opt([p])\subseteq\Opt([q])$. The other containment follows by symmetry.
\end{proof}

\begin{proposition}
\label{prop:unique simplification}
For every rulegraph $\mathsf{R}$ there is a unique congruence relation $\bowtie$ on $V(\mathsf{R})$ such that $\mathsf{R}/{\bowtie}$ is simple. 
\end{proposition}

\begin{proof}
The uniqueness of $\bowtie$ follows from Proposition~\ref{prop:uniqueSlim}. To show existence, we construct $\bowtie$ by transfinite construction with respect to the formal birthday of the positions. At each step we identify all positions with formal birthday $\xi$ that have the same options.
Let $V_\xi:=\{p\in V(\mathsf{R})\mid \fbd(p)=\xi\}$ be the set of positions with formal birthday $\xi$ and let $V_{\le\xi}:=\bigcup_{\zeta\le\xi}V_\zeta$. Note that $V(\mathsf{R})$ is the direct limit $\varinjlim V_{\le\xi}$. We build the congruence relation $\bowtie$ as a direct limit ${\bowtie}:=\varinjlim{\bowtie}_\xi$, where $\bowtie_\xi$ is a congruence relation on $V_{\le\xi}$ satisfying ${\bowtie}_\xi\rvert_{V_{\zeta}}= {\bowtie_{\zeta}}$ for all $\zeta<\xi$. For $p,q\in V_\xi$, we let $p\bowtie_\xi q$ exactly when $[\Opt(p)]_\zeta=[\Opt(q)]_\zeta$ for some $\zeta<\xi$. 
The construction of $\bowtie$ guarantees that the option map $\Opt$ is injective on $\mathsf{R}/{\bowtie}$. Hence $\mathsf{R}/{\bowtie}$ is simple by Proposition~\ref{prop:injective}.
\end{proof}

\begin{remark}
\label{rem:altBowtie}
We sketch an alternative, though nonconstructive way to prove the previous proposition. Define $p\bowtie q$ if and only if $p\sim q$ for some congruence relation $\sim$. Then it is not hard to check that the rulegraph $\mathsf{R}/{\bowtie}$ is simple. The only step in the proof that is not so straightforward is to show that $\bowtie$ is transitive. This follows from~\cite[Lemma~2.4]{EmulationalEquivalence} stating that the transitive closure of the union of two congruence relations is also a congruence relation.
\end{remark}

We call the rulegraph $\mathsf{R}/{\bowtie}$ the \emph{minimum quotient} of $\mathsf{R}$. The corresponding quotient map is an example of a surjective option preserving rulegraph map from $\mathsf{R}$ to a simple rulegraph. Such a map is essentially unique. 
In fact, if $\alpha:\mathsf{R}\to \mathsf{S}$ is a surjective option preserving rulegraph map and $\mathsf{S}$ is simple, then $\mathsf{S}$ is isomorphic to $\mathsf{R}/{\bowtie}$.

\begin{example}
Figure~\ref{fig:slimificationInf} shows an option preserving map from the gamegraph $\mathsf{G}$ consisting of the infinite forward paths with $n$ vertices for $n\in\mathbb{N}$ with their source vertices identified to the simple gamegraph $\mathsf{S}$. The coloring and shapes indicate the source and option preserving map $\alpha:\mathsf{G}\to\mathsf{S}$. Note that $\mathsf{S}$ is simple since $\Opt$ is injective. Thus, $\mathsf{G}/{\bowtie}$ is isomorphic to $\mathsf{S}$.
\end{example}

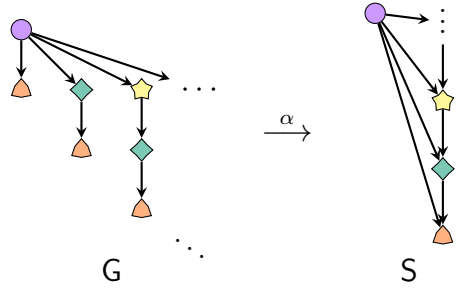
\begin{figure}[ht]
\centering
\begin{tikzpicture}[scale=.8,rotate=-90]
\node (00) [small vert,fill=purple!50] at (0,0) {};
\node (10) [small vert,fill=orange!50,star,star points=3] at (1,0) {};
\node (11) [small vert,fill=springgreen!50,star,star points=4] at (1,1) {};
\node (21) [small vert,fill=orange!50,star,star points=3] at (2,1) {};
\node (12) [small vert,fill=yellow!50,star,star points=5] at (1,2) {};
\node (22) [small vert,fill=springgreen!50,star,star points=4] at (2,2) {};
\node (32) [small vert,fill=orange!50,star,star points=3] at (3,2) {};
\node (dts) at (1,3) {$\cdots$};
\node at (3.5,2.8) {$\ddots$};
\path [d] (00) to (10);
\path [d] (00) to (11);
\path [d] (00) to (12);
\path [d] (11) to (21);
\path [d] (12) to (22);
\path [d] (22) to (32);
\path [d] (00) to (dts);  
\node at (4,1.5) {$\mathsf{G}$};
\end{tikzpicture}
\hspace{.1cm}
\raisebox{2cm}{\hbox{$\overset\alpha\longrightarrow$}}
$\quad$
\begin{tikzpicture}[scale=.9,rotate=-90]
\node (0) at (5.2,0) {$\vdots$};
\node (1) at (5.5,0) {};
\node (2) [small vert,fill=yellow!50,star,star points=5] at (6.5,0) {};
\node (3) [small vert,fill=springgreen!50,star,star points=4] at (7.5,0) {};
\node (4) [small vert,fill=orange!50,star,star points=3] at (8.5,0) {};
\node (s) [small vert,fill=purple!50] at (5.2,-1) {};

\path [d] (1) to (2);
\path [d] (2) to (3);
\path [d] (3) to (4);

\path [d] (s) to (2);
\path [d] (s) to (3);
\path [d] (s) to (4);

\path [d] (s) to (5.3,-.2);
\node at (9,-0.5) {$\mathsf{S}$};
\end{tikzpicture}

\caption{
\label{fig:slimificationInf}
An option and source preserving map from a game $\mathsf{G}$ to a simple game $\mathsf{S}$.
}
\end{figure}

\begin{example}
In Example~\ref{exa:NIM}, we introduced two different rulegraphs for the collection of all finite one-pile NIM gamegraphs. The second rulegraph is isomorphic to the minimum quotient of the first.
\end{example}

In practice we can find $\bowtie$ by recursively identifying positions that have the same options until no two such positions exist. The next example illustrates a consequence of Proposition~\ref{prop:uniqueSlim},  that the minimum quotient of a rulegraph is independent of the order in which one chooses to identify compatible positions.

\begin{example}\label{ex:slimifaction}
Figure~\ref{fig:slimification} shows one possible process to find the minimum quotient of a rulegraph through identifications indicated in red stars.
Since the identification at each step creates only one nontrivial congruence class, each position in this class must have exactly the same option set as explained in Remark~\ref{rem:mustEqual}. Note that in the first identification step we could have identified all three positions in the top row.
\end{example}

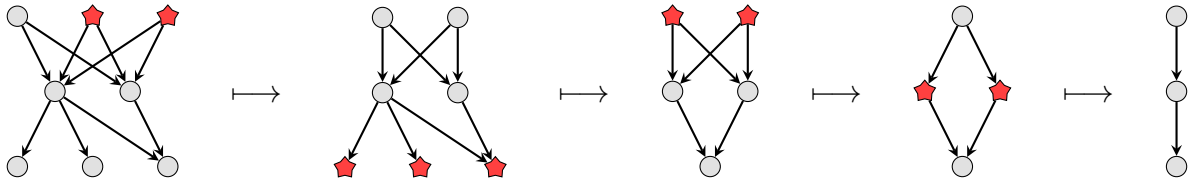
\begin{figure}[ht]

\begin{tikzpicture}[]
\node (-12) [small vert] at (-1,2) {};
\node (02) [small vert,fill=red!75,star,star points=5] at (0,2) {};
\node (12) [small vert,fill=red!75,star,star points=5] at (1,2) {};
\node (01) [small vert] at (-.5,1) {};
\node (11) [small vert] at (.5,1) {};
\node (00) [small vert] at (0,0) {};
\node (10) [small vert] at (1,0) {};
\node (-10) [small vert] at (-1,0) {};
\path [d] (-12) to (01);
\path [d] (-12) to (11);
\path [d] (02) to (01);
\path [d] (02) to (11);
\path [d] (12) to (01);
\path [d] (12) to (11);
\path [d] (01) to (00);
\path [d] (01) to (10);
\path [d] (11) to (10);
\path [d] (01) to (-10);
\end{tikzpicture}
$\quad$
\raisebox{1cm}{\hbox{$\longmapsto$}}
$\quad$
\begin{tikzpicture}
\node (-12) [small vert] at (2.5,2) {};
\node (02) [small vert] at (3.5,2) {};
\node (01) [small vert] at (2.5,1) {};
\node (11) [small vert] at (3.5,1) {};
\node (00) [small vert,fill=red!75,star,star points=5] at (3,0) {};
\node (10) [small vert,fill=red!75,star,star points=5] at (4,0) {};
\node (-10) [small vert, fill=red!75,star,star points=5] at (2,0) {};
\path [d] (-12) to (01);
\path [d] (-12) to (11);
\path [d] (02) to (01);
\path [d] (02) to (11);
\path [d] (01) to (00);
\path [d] (01) to (10);
\path [d] (11) to (10);
\path [d] (01) to (-10);
\end{tikzpicture}
$\quad$
\raisebox{1cm}{\hbox{$\longmapsto$}}
$\quad$
\begin{tikzpicture}
\node (-12) [small vert,fill=red!75,star,star points=5] at (6,2) {};
\node (02) [small vert,fill=red!75,star,star points=5] at (7,2) {};
\node (01) [small vert] at (6,1) {};
\node (11) [small vert] at (7,1) {};
\node (00) [small vert] at (6.5,0) {};
\path [d] (-12) to (01);
\path [d] (-12) to (11);
\path [d] (02) to (01);
\path [d] (02) to (11);
\path [d] (01) to (00);
\path [d] (11) to (00);
\end{tikzpicture}
$\quad$
\raisebox{1cm}{\hbox{$\longmapsto$}}
$\quad$
\begin{tikzpicture}
\node (02) [small vert] at (6.5,2) {};
\node (01) [small vert,fill=red!75,star,star points=5] at (6,1) {};
\node (11) [small vert,fill=red!75,star,star points=5] at (7,1) {};
\node (00) [small vert] at (6.5,0) {};
\path [d] (02) to (01);
\path [d] (02) to (11);
\path [d] (01) to (00);
\path [d] (11) to (00);
\end{tikzpicture}
$\quad$
\raisebox{1cm}{\hbox{$\longmapsto$}}
$\quad$
\begin{tikzpicture}
\node (02) [small vert] at (9,2) {};
\node (01) [small vert] at (9,1) {};
\node (00) [small vert] at (9,0) {};
\path [d] (02) to (01);
\path [d] (01) to (00);
\end{tikzpicture}

\caption{
\label{fig:slimification}
Finding the minimum quotient through a sequence of identifications indicated in red stars.  
}
\end{figure}

\begin{example}
The sum of two simple gamegraphs is not simple unless one of the game\-graphs is trivial. A terminal position in one gamegraph paired with a position whose only options are terminal in the other gamegraph gives a position whose only options are terminal in the sum. These are all identified in the minimum quotient. For example, Figure~\ref{fig:sum of slims} shows $\star 1$, $\star 2$, $\star 1+\star 2$, and the minimum quotient $(\star 1+\star 2)/{\bowtie}$.
\end{example}

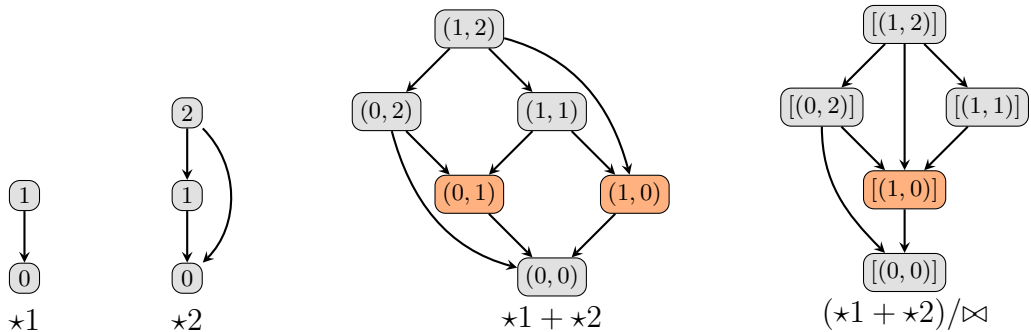
\begin{figure}[ht]
\begin{tikzpicture}[scale=1.1]
\node (a) [rect vert] at (0,1) {\scriptsize 1};
\node (b) [rect vert] at (0,0) {\scriptsize 0};
\path[d] (a) to (b);
\node at (0,-.5) {$\star 1$};
\end{tikzpicture}
\hfil
\begin{tikzpicture}[scale=1.1]
\node (a) [rect vert] at (0,2) {\scriptsize 2};
\node (b) [rect vert] at (0,1) {\scriptsize 1};
\node (c) [rect vert] at (0,0) {\scriptsize 0};
 \path[d] (a) to (b);
 \path[d] (b) to (c);
 \path[d,out=-45,in=45] (a) to (c);
 \node at (0,-.5) {$\star 2$};
 \end{tikzpicture}
 \hfil
 \begin{tikzpicture}[scale=1.1]
 \node (12) [rect vert] at (0,3) {\scriptsize $(1,2)$};
 \node (11) [rect vert] at (1,2) {\scriptsize $(1,1)$};
 \node (10) [rect vert,fill=orange!50] at (2,1) {\scriptsize $(1,0)$};
 \node (02) [rect vert] at (-1,2) {\scriptsize $(0,2)$};
 \node (01) [rect vert,fill=orange!50] at (0,1) {\scriptsize $(0,1)$};
 \node (00) [rect vert] at (1,0) {\scriptsize $(0,0)$};
 \path[d] (12) to (02);
 \path[d] (12) to (11);
 \path[d,out=-15,in=105] (12) to (10);
 \path[d] (11) to (10);
 \path[d] (02) to (01);
 \path[d] (01) to (00);
 \path[d,out=-75,in=-195] (02) to (00);
 \path[d] (11) to (01);
 \path[d] (10) to (00);
 \node at (1,-.5) {$\star 1+\star 2$};
 \end{tikzpicture}
 \hfil
 \begin{tikzpicture}[scale=1.1]
 \node (12) [rect vert] at (0,3) {\scriptsize $[(1,2)]$};
 \node (11) [rect vert] at (1,2) {\scriptsize $[(1,1)]$};
 \node (10) [rect vert,fill=orange!50] at (0,1) {\scriptsize $[(1,0)]$};
 \node (02) [rect vert] at (-1,2) {\scriptsize $[(0,2)]$};
 \node (00) [rect vert] at (0,0) {\scriptsize $[(0,0)]$};
 \path[d] (12) to (02);
 \path[d] (12) to (11);
 \path[d] (12) to (10);
 \path[d] (11) to (10);
 \path[d] (02) to (10);
 \path[d,out=-90] (02) to (00);
 \path[d] (10) to (00);
 \node at (0,-.5) {$(\star 1+\star 2)/{\bowtie}$};
 \end{tikzpicture}
 \hfil
\caption{
\label{fig:sum of slims}
Two simple gamegraphs whose sum is not simple.}
\end{figure}

\begin{remark}
\label{rem:slim equal traditional}
An impartial game is often defined \cite{albert2007lessons, ONAG,SiegelBook} recursively as a set $G$ containing its options, where options are games themselves. This means that a game is identified with the set of its options. To see that this traditional definition is a special case of our definition, we can recursively build a digraph that consists of arrows of the form $(P,Q)$ for $Q\in P$ where $P$ is a position of $G$. This digraph is a gamegraph $\mathsf{G}$. It is clear that the $\Opt$ function is injective on $\mathsf{G}$, and so $\mathsf{G}$ is simple by Proposition~\ref{prop:injective}. Conversely, in a given gamegraph $\mathsf{G}$ we can recursively replace each position $p$ with the set $P:=\Opt(p)$ of its options and the resulting set corresponding to the starting position is a game $G$. Note that this replacement also applies to the arrows. In fact, the minimum quotient of a gamegraph is essentially the result of this replacement process.
\end{remark}

\begin{example}
Figure~\ref{fig;identification} shows the process that recursively identifies each position of a gamegraph with the set of its options until a simple gamegraph is created. The starting position of the resulting gamegraph is $\{\emptyset,\{\emptyset\}\}$, which we can recognize as the traditional $*2$ game. Note that this process will not always yield a nimber. For example, applying this process to the rightmost gamegraph in Figure~\ref{fig:slimification} yields the starting position $\{\{\emptyset\}\}$, which is not a nimber but has nim-number 0.
\end{example}

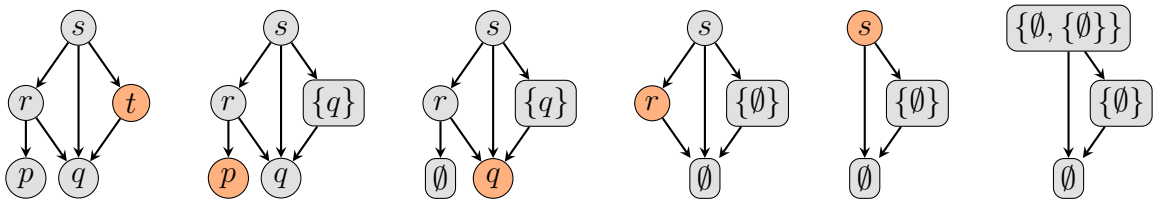
\begin{figure}[ht]
\centering
\begin{tikzpicture}[xscale=.7]
\node (02) [small vert] at (1,2) {$s$};
\node (01) [small vert] at (0,1) {$r$};
\node (11) [small vert,fill=orange!50] at (2,1) {$t$};
\node (00) [small vert] at (0,0) {$p$};
\node (10) [small vert] at (1,0) {$q$};
\path [d] (02) to (01);
\path [d] (02) to (11);
\path [d] (01) to (00);
\path [d] (01) to (10);
\path [d] (11) to (10);
\path [d] (02) to (10);
\end{tikzpicture}
\hfil
\begin{tikzpicture}[xscale=.7]
\node (02) [small vert] at (1,2) {$s$};
\node (01) [small vert] at (0,1) {$r$};
\node (11) [rect vert] at (2,1) {$\{q\}$};
\node (00) [small vert,fill=orange!50] at (0,0) {$p$};
\node (10) [small vert] at (1,0) {$q$};
\path [d] (02) to (01);
\path [d] (02) to (11);
\path [d] (01) to (00);
\path [d] (01) to (10);
\path [d] (11) to (10);
\path [d] (02) to (10);
\end{tikzpicture}
\hfil
\begin{tikzpicture}[xscale=.7]
\node (02) [small vert] at (1,2) {$s$};
\node (01) [small vert] at (0,1) {$r$};
\node (11) [rect vert] at (2,1) {$\{q\}$};
\node (00) [rect vert] at (0,0) {$\emptyset$};
\node (10) [small vert,fill=orange!50] at (1,0) {$q$};
\path [d] (02) to (01);
\path [d] (02) to (11);
\path [d] (01) to (00);
\path [d] (01) to (10);
\path [d] (11) to (10);
\path [d] (02) to (10);
\end{tikzpicture}
\hfil
\begin{tikzpicture}[xscale=.7]
\node (02) [small vert] at (1,2) {$s$};
\node (01) [small vert,fill=orange!50] at (0,1) {$r$};
\node (11) [rect vert] at (2,1) {$\{\emptyset\}$};
\node (10) [rect vert] at (1,0) {$\emptyset$};
\path [d] (02) to (01);
\path [d] (02) to (11);
\path [d] (01) to (10);
\path [d] (11) to (10);
\path [d] (02) to (10);
\end{tikzpicture}
\hfil
\begin{tikzpicture}[xscale=.7]
\node (02) [small vert,fill=orange!50] at (1,2) {$s$};
\node (11) [rect vert] at (2,1) {$\{\emptyset\}$};
\node (10) [rect vert] at (1,0) {$\emptyset$};
\path [d] (02) to (11);
\path [d] (11) to (10);
\path [d] (02) to (10);
\end{tikzpicture}
\hfil
\begin{tikzpicture}[xscale=.7]
\node (02) [rect vert] at (1,2) {$\{\emptyset,\{\emptyset\}\}$};
\node (11) [rect vert] at (2,1) {$\{\emptyset\}$};
\node (10) [rect vert] at (1,0) {$\emptyset$};
\path [d] (02) to (11);
\path [d] (11) to (10);
\path [d] (02) to (10);
\end{tikzpicture}

\caption{
\label{fig;identification}
The process of recursive identification of positions with their set of options. Each step shows the identification at the orange position. 
}
\end{figure}


Extending the definition for gamegraphs in \cite{EmulationalEquivalence} to rulegraphs, we say that rulegraphs $\mathsf{R}$ and $\mathsf{S}$ are \emph{emulationally equivalent} if there are congruence relations $\sim_1$ and $\sim_2$ such that $\mathsf{R}/{\sim_1}$ and $\mathsf{S}/{\sim_2}$ are isomorphic.

The next result follows immediately from Proposition~\ref{prop:fundamental homomorphism theorem}.

\begin{proposition}
\label{prop:equivalence between emulationally equivalent and option preserving}
Two rulegraphs $\mathsf{R}$ and $\mathsf{S}$ are emulationally equivalent if and only if there are surjective option preserving rulegraph maps $\alpha:\mathsf{R}\to \mathsf{T}$ and $\beta:\mathsf{S}\to \mathsf{T}$ for some rulegraph $\mathsf{T}$.
\end{proposition}

\begin{corollary}
Two rulegraphs are emulationally equivalent if and only if their minimum quotients are isomorphic.
\end{corollary}

\begin{example}
Figure~\ref{fig:DaniejelasCounterexample} shows that emulational equivalence does not imply the existence of an option preserving map. It is an easy exercise to check that the shape and color-coded maps $\alpha:\mathsf{G}\to \mathsf{K}$ and $\beta:\mathsf{H}\to \mathsf{K}$ are surjective and option preserving. 

Proposition \ref{prop:nim preserving} implies that option preserving maps preserve formal birthdays. Therefore, any such map from $\mathsf{G}$ to $\mathsf{H}$ has to map the green diamond from $\mathsf{G}$ to one of the two green diamonds in $\mathsf{H}$. Neither of these choices work since it is impossible to define the images of the blue 6-sided star and the purple 5-sided stars while maintaining the option preserving property.

An option preserving map from $\mathsf{H}$ to $\mathsf{G}$ would have to map both green diamonds from $\mathsf{H}$ to green diamond in $\mathsf{G}$ and both orange terminal positions to the left orange terminal position in $\mathsf{G}$. Then both purple 5-sided stars from $\mathsf{H}$ would map onto the left purple 5-sided star in $\mathsf{G}$. The red circle then breaks the rule for an option preserving map. 
\end{example}

\begin{figure}[ht]
\centering
\begin{tikzpicture}[scale=.8]
\node (G) at (1,-.5) {$\mathsf{G}$};
\node (a) [small vert,fill=rose!75] at (1,3) {};
\node (b1) [small vert,fill=purple!75,star,star points=5] at (0,2) {};
\node (b2) [small vert,fill=icyblue!75,star,star points=6] at (1,2) {};
\node (b3) [small vert,fill=purple!75,star,star points=5] at (2,2) {};
\node (c) [small vert,fill=springgreen!75,star,star points=4] at (.5,1) {};
\node (d1) [small vert,fill=orange!75,star,star points=3] at (.5,0) {};
\node (d2) [small vert,fill=orange!75,star,star points=3] at (1.5,0) {};
\path [d] (a) to (b1);
\path [d] (a) to (b2);
\path [d] (a) to (b3);
\path [d] (b1) to (c);
\path [d] (b2) to (c);
\path [d,out=-90,in=125] (b1) to (d1);
\path [d] (b3) to (c);
\path [d] (b3) to (d2);
\path [d] (c) to (d1);
\end{tikzpicture}
\quad
\raisebox{1.8cm}{\hbox{$\overset\alpha\longrightarrow$}}
\quad
\begin{tikzpicture}[scale=.8]
\node (K) at (4.5,-.5) {$\mathsf{K}$};
\node (1) [small vert,fill=rose!75] at (4.5,3) {};
\node (21) [small vert,fill=icyblue!75,star,star points=6] at (4,2) {};
\node (22) [small vert,fill=purple!75,star,star points=5] at (5,2) {};
\node (3) [small vert,fill=springgreen!75,star,star points=4] at (4.5,1) {};
\node (4) [small vert,fill=orange!75,star,star points=3] at (4.5,0) {};
\path [d] (1) to (21);
\path [d] (1) to (22);
\path [d] (21) to (3);
\path [d] (22) to (3);
\path [d,out=-90,in=55] (22) to (4);
\path [d] (3) to (4);
\end{tikzpicture}
\quad
\raisebox{1.8cm}{\hbox{$\overset\beta\longleftarrow$}}
\quad
\begin{tikzpicture}[scale=.8]

\node (H) at (8,-.5) {$\mathsf{H}$};
\node (s1) [small vert,fill=rose!75] at (8,3) {};
\node (s21) [small vert,fill=icyblue!75,star,star points=6] at (7,2) {};
\node (s22) [small vert,fill=purple!75,star,star points=5] at (8,2) {};
\node (s23) [small vert,fill=purple!75,star,star points=5] at (9,2) {};
\node (s31) [small vert,fill=springgreen!75,star,star points=4] at (7.5,1) {};
\node (s32) [small vert,fill=springgreen!75,star,star points=4] at (8.5,1) {};
\node (s41) [small vert,fill=orange!75,star,star points=3] at (7.5,0) {};
\node (s42) [small vert,fill=orange!75,star,star points=3] at (8.5,0) {};
\path [d] (s1) to (s21);
\path [d] (s1) to (s22);
\path [d] (s1) to (s23);
\path [d] (s21) to (s31);
\path [d] (s22) to (s32);
\path [d,out=-90,in=55] (s22) to (s41);
\path [d] (s23) to (s32);
\path [d,out=-90,in=55] (s23) to (s42);
\path [d] (s31) to (s41);
\path [d] (s32) to (s42);
\end{tikzpicture}
\caption{
\label{fig:DaniejelasCounterexample}
Two gamegraphs and their isomorphic minimum quotients.}
\end{figure}
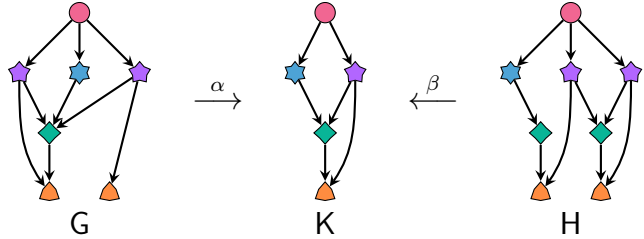

Emulational equivalence is an equivalence relation on the class of rulegraphs (respectively, gamegraphs). The simple rulegraphs (respectively, gamegraphs) form a class of unique representatives up to isomorphism. The unique representatives of gamegraphs are essentially those in Remark~\ref{rem:slim equal traditional}. 

\begin{proposition}
The set $\Con(\mathsf{R})$ of congruence relations on a rulegraph $\mathsf{R}$ forms a complete lattice under inclusion.
\end{proposition}

\begin{proof}
Let $\mathcal{A}\subseteq \Con(\mathsf{R})$.
It is easy to see that the meet of $\mathcal{A}$ is $\bigwedge\mathcal{A}:=\bigcap\mathcal{A}$. Let $\mathcal{B}:=\{C\in\Con(\mathsf{R}) \mid \bigcup\mathcal{A}\subseteq C\}$. Note that $\mathcal{B}$ is not empty since it contains $\bowtie$. It is easy to see that the join of $\mathcal{A}$ is $\bigvee\mathcal{A}:=\bigwedge\mathcal{B}$. 
\end{proof}

Note that $\mathsf{R}$ is simple if and only if $\Con(\mathsf{R})$ contains only one congruence relation $\Delta={\bowtie}$. This is analogous to the universal algebra result stating that a universal algebra is simple if and only if its only congruence relations are $\Delta$ and $\nabla$.

The following result can be thought of as the Fourth Isomorphism Theorem for rulegraphs. The proof is left to the reader.

\begin{proposition}\label{prop:fourth iso theorem}
Let $\mathsf{R}$ be a rulegraph and $D\in\Con(\mathsf{R})$. Then 
\[
C\mapsto C/D:\{C\in\Con(\mathsf{R})\mid D\subseteq C \}\to\Con(\mathsf{R}/D)
\]
is a lattice isomorphism.
\end{proposition}

\begin{example}
Figure~\ref{fig:4thIso} demonstrates the Fourth Isomorphism Theorem. 
Congruence relations are denoted by a list of the nontrivial congruence classes separated by bars. The congruence relation $D=56$ identifies positions $5$ and $6$. The maximum congruence relation is ${\bowtie} = 12|34|56$.
\end{example}

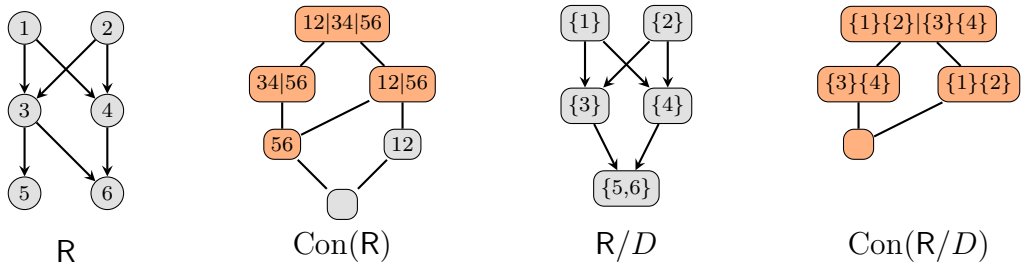
\begin{figure}
\begin{tikzpicture}[scale=1.1]
\node (02) [small vert] at (0,2) {$\scriptstyle 1$};
\node (12) [small vert] at (1,2) {$\scriptstyle  2$};
\node (01) [small vert] at (0,1) {$\scriptstyle  3$};
\node (11) [small vert] at (1,1) {$\scriptstyle  4$};
\node (00) [small vert] at (0,0) {$\scriptstyle  5$};
\node (10) [small vert] at (1,0) {$\scriptstyle  6$};
\path [d] (02) to (01);
\path [d] (02) to (11);
\path [d] (12) to (01);
\path [d] (12) to (11);
\path [d] (01) to (00);
\path [d] (01) to (10);
\path [d] (11) to (10);
\node  at (.5,-.7) {$\mathsf{R}$};
\end{tikzpicture}
\hfil
\begin{tikzpicture}[scale=.8]
\node (c12-34-56) [rect vert, fill=orange!50] at (0,3) {$\scriptstyle 12|34|56$};
\node (c34-56) [rect vert, fill=orange!50] at (-1,2) {$\scriptstyle 34|56$};
\node (c12-56) [rect vert, fill=orange!50] at (1,2) {$\scriptstyle  12|56$};
\node (c12) [rect vert] at (1,1) {$\scriptstyle  12$};
\node (c56) [rect vert, fill=orange!50] at (-1,1) {$\scriptstyle  56$};
\node (id) [rect vert] at (0,0) {$\scriptstyle  $};
\path [d,-] (c12-34-56) to (c34-56);
\path [d,-] (c12-34-56) to (c12-56);
\path [d,-] (c34-56) to (c56);
\path [d,-] (c12-56) to (c12);
\path [d,-] (c12-56) to (c56);
\path [d,-] (c56) to (id);
\path [d,-] (c12) to (id);
\node  at (0,-.7) {$\Con(\mathsf{R})$};
\end{tikzpicture}
\hfil
\begin{tikzpicture}[scale=1.1]
\node (02) [rect vert] at (0,2) {$\scriptstyle \{1\}$};
\node (12) [rect vert] at (1,2) {$\scriptstyle  \{2\}$};
\node (01) [rect vert] at (0,1) {$\scriptstyle  \{3\}$};
\node (11) [rect vert] at (1,1) {$\scriptstyle  \{4\}$};
\node (10) [rect vert] at (.5,0) {$\scriptstyle  \{5,6\}$};
\path [d] (02) to (01);
\path [d] (02) to (11);
\path [d] (12) to (01);
\path [d] (12) to (11);
\path [d] (01) to (10);
\path [d] (11) to (10);
\node  at (.5,-.7) {$\mathsf{R}/D$};
\end{tikzpicture}
\hfil
\begin{tikzpicture}[scale=.8]
\node (c12-34-56) [rect vert,fill=orange!50] at (0,3) {$\scriptstyle \{1\}\{2\}|\{3\}\{4\}$};
\node (c34-56) [rect vert,fill=orange!50] at (-1,2) {$\scriptstyle \{3\}\{4\}$};
\node (c12-56) [rect vert,fill=orange!50] at (1,2) {$\scriptstyle  \{1\}\{2\}$};
\node (c56) [rect vert,fill=orange!50] at (-1,1) {$\scriptstyle  $};
\path [d,-] (c12-34-56) to (c34-56);
\path [d,-] (c12-34-56) to (c12-56);
\path [d,-] (c34-56) to (c56);
\path [d,-] (c12-56) to (c56);
\node  at (0,-.7) {$\Con(\mathsf{R}/D)$};
\end{tikzpicture}

\caption{
\label{fig:4thIso}
The lattice of congruence relations for a rulegraph $\mathsf{R}$ and its quotient rulegraph $\mathsf{R}/\mathsf{D}$. 
}
\end{figure}

\section{Enumerating simple rulegraphs and gamegraphs}\label{sec:enumeration}

In this section, we explore two possible enumerations of simple gamegraphs and rulegraphs: first by formal birthday, and then by number of positions.

\subsection{Counting by formal birthday}

The \emph{rank} of a set is the least ordinal greater than the rank of each element of the set. An impartial game, as defined in \cite{albert2007lessons,SiegelBook}, with formal birthday $n$ is a set with rank $n$. So the number of impartial games with formal birthday at most $n$ is ${^{n}2}:=2^{2^{\cdots 2}}$ with $n$ copies of 2, as described in \cite[A014221]{oeis}.  Note that ${^{0}2=1}$ and ${^{n+1}2}=2^{(^n2)}$. Since the correspondence described in Remark~\ref{rem:slim equal traditional} is a formal birthday preserving bijection, we have the following result. 

\begin{proposition}\label{prop:number of rulegraphs}
The number of simple gamegraphs up to isomorphism with formal birthday at most $n$ is ${^{n}2}$.
\end{proposition}

\begin{example}
Figure \ref{fig:small game graphs} shows the 16 simple gamegraphs up to isomorphism with formal birthday at most 3.
\end{example}

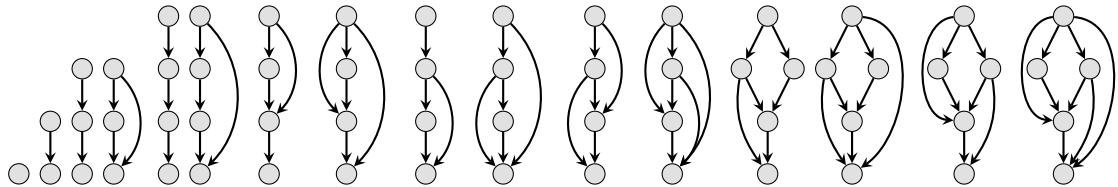
\begin{figure}[ht]
\centering
\begin{tikzpicture}[scale=.7]
\node (a) [small vert] at (0,1) {};
\end{tikzpicture}
\begin{tikzpicture}[scale=.7]
\node (a) [small vert] at (0,1) {};
\node (b) [small vert] at (0,0) {};
\path[d] (a) to (b);
\end{tikzpicture}
\begin{tikzpicture}[scale=.7]
\node (a) [small vert] at (0,2) {};
\node (b) [small vert] at (0,1) {};
\node (c) [small vert] at (0,0) {};
\path[d] (a) to (b);
\path[d] (b) to (c);
\end{tikzpicture}
\begin{tikzpicture}[scale=.7]
\node (a) [small vert] at (0,2) {};
\node (b) [small vert] at (0,1) {};
\node (c) [small vert] at (0,0) {};
\path[d] (a) to (b);
\path[d] (b) to (c);
\path[d,out=-45,in=45] (a) to (c);
\end{tikzpicture}
\begin{tikzpicture}[scale=.7]
\node (a) [small vert] at (0,3) {};
\node (b) [small vert] at (0,2) {};
\node (c) [small vert] at (0,1) {};
\node (d) [small vert] at (0,0) {};
\path[d] (a) to (b);
\path[d] (b) to (c);
\path[d] (c) to (d);
\end{tikzpicture}
\begin{tikzpicture}[scale=.7]
\node (a) [small vert] at (0,3) {};
\node (b) [small vert] at (0,2) {};
\node (c) [small vert] at (0,1) {};
\node (d) [small vert] at (0,0) {};
\path[d] (a) to (b);
\path[d] (b) to (c);
\path[d] (c) to (d);
\path[d,out=-45,in=45] (a) to (d);
\end{tikzpicture}
\begin{tikzpicture}[scale=.7]
\node (a) [small vert] at (0,3) {};
\node (b) [small vert] at (0,2) {};
\node (c) [small vert] at (0,1) {};
\node (d) [small vert] at (0,0) {};
\path[d] (a) to (b);
\path[d] (b) to (c);
\path[d] (c) to (d);
\path[d,out=-45,in=45] (a) to (c);
\end{tikzpicture}
\begin{tikzpicture}[scale=.7]
\node (a) [small vert] at (0,3) {};
\node (b) [small vert] at (0,2) {};
\node (c) [small vert] at (0,1) {};
\node (d) [small vert] at (0,0) {};
\path[d] (a) to (b);
\path[d] (b) to (c);
\path[d] (c) to (d);
\path[d,out=-45,in=45] (a) to (d);
\path[d,out=-135,in=135] (a) to (c);
\end{tikzpicture} \vphantom{$\int_\int$} 
\begin{tikzpicture}[scale=.7]
\node (a) [small vert] at (0,3) {};
\node (b) [small vert] at (0,2) {};
\node (c) [small vert] at (0,1) {};
\node (d) [small vert] at (0,0) {};
\path[d] (a) to (b);
\path[d] (b) to (c);
\path[d] (c) to (d);
\path[d,out=-45,in=45] (b) to (d);
\end{tikzpicture}
\begin{tikzpicture}[scale=.7]
\node (a) [small vert] at (0,3) {};
\node (b) [small vert] at (0,2) {};
\node (c) [small vert] at (0,1) {};
\node (d) [small vert] at (0,0) {};
\path[d] (a) to (b);
\path[d] (b) to (c);
\path[d] (c) to (d);
\path[d,out=-135,in=135] (b) to (d);
\path[d,out=-45,in=45] (a) to (d);
\end{tikzpicture}
\begin{tikzpicture}[scale=.7]
\node (a) [small vert] at (0,3) {};
\node (b) [small vert] at (0,2) {};
\node (c) [small vert] at (0,1) {};
\node (d) [small vert] at (0,0) {};
\path[d] (a) to (b);
\path[d] (b) to (c);
\path[d] (c) to (d);
\path[d,out=-135,in=135] (b) to (d);
\path[d,out=-45,in=45] (a) to (c);
\end{tikzpicture}
\begin{tikzpicture}[scale=.7]
\node (a) [small vert] at (0,3) {};
\node (b) [small vert] at (0,2) {};
\node (c) [small vert] at (0,1) {};
\node (d) [small vert] at (0,0) {};
\path[d] (a) to (b);
\path[d] (b) to (c);
\path[d] (c) to (d);
\path[d,out=-135,in=135] (a) to (c);
\path[d,out=-45,in=45] (a) to (d);
\path[d,out=-45,in=45] (b) to (d);
\end{tikzpicture}
\begin{tikzpicture}[scale=.7]
\node (a) [small vert] at (.5,3) {};
\node (b1) [small vert] at (0,2) {};
\node (b2) [small vert] at (1,2) {};
\node (c) [small vert] at (.5,1) {};
\node (d) [small vert] at (.5,0) {};
\path [d] (a) to (b1);
\path [d] (a) to (b2);
\path [d] (b1) to (c);
\path [d] (b2) to (c);
\path [d,out=-100,in=125] (b1) to (d);
\path [d] (c) to (d);
\end{tikzpicture}
\begin{tikzpicture}[scale=.7]
\node (a) [small vert] at (.5,3) {};
\node (b1) [small vert] at (0,2) {};
\node (b2) [small vert] at (1,2) {};
\node (c) [small vert] at (.5,1) {};
\node (d) [small vert] at (.5,0) {};
\path [d] (a) to (b1);
\path [d] (a) to (b2);
\path [d] (b1) to (c);
\path [d] (b2) to (c);
\path [d,out=-100,in=125] (b1) to (d);
\path [d,out=-5,in=35] (a) to (d);
\path [d] (c) to (d);
\end{tikzpicture}
\hspace{-1em}
\begin{tikzpicture}[scale=.7]
\node (a) [small vert] at (.5,3) {};
\node (b1) [small vert] at (0,2) {};
\node (b2) [small vert] at (1,2) {};
\node (c) [small vert] at (.5,1) {};
\node (d) [small vert] at (.5,0) {};
\path [d] (a) to (b1);
\path [d] (a) to (b2);
\path [d] (b1) to (c);
\path [d] (b2) to (c);
\path [d,out=-175,in=175] (a) to (c);
\path [d,out=-80,in=55] (b2) to (d);
\path [d] (c) to (d);
\end{tikzpicture}
\begin{tikzpicture}[scale=.7]
\node (a) [small vert] at (.5,3) {};
\node (b1) [small vert] at (0,2) {};
\node (b2) [small vert] at (1,2) {};
\node (c) [small vert] at (.5,1) {};
\node (d) [small vert] at (.5,0) {};
\path [d] (a) to (b1);
\path [d] (a) to (b2);
\path [d] (b1) to (c);
\path [d] (b2) to (c);
\path [d,out=-80,in=55] (b2) to (d);
\path [d,out=-5,in=35] (a) to (d);
\path [d,out=-175,in=175] (a) to (c);
\path [d] (c) to (d);
\end{tikzpicture}

\caption{Simple gamegraphs up to isomorphism with formal birthday at most $3$.
\label{fig:small game graphs}}
\end{figure}

Our next goal is to determine the number $x_d$ of simple rulegraphs with formal birthday $d$ up to isomorphism.

\begin{example}
\label{eg:small_ruleset}
Figure \ref{fig:small ruleset graphs 1} shows the simple rulegraphs up to isomorphism with formal birthday at most 2. Note that $x_0=1$, $x_1=1$, and $x_2=3$.
\end{example}

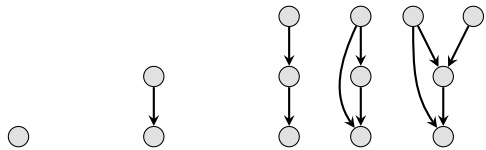
\begin{figure}[ht]
\centering
\begin{tikzpicture}[scale=.8]
\node (d1) [small vert] at (.5,0) {};
\end{tikzpicture}
$\qquad\quad$
\begin{tikzpicture}[scale=.8]
\node (c) [small vert] at (.5,1) {};
\node (d1) [small vert] at (.5,0) {};
\path [d] (c) to (d1);
\end{tikzpicture}
$\qquad\quad$
\begin{tikzpicture}[scale=.8]
\node (b1) [small vert] at (0.5,2) {};
\node (c) [small vert] at (.5,1) {};
\node (d1) [small vert] at (.5,0) {};
\path [d] (b1) to (c);
\path [d] (c) to (d1);
\end{tikzpicture}
$\ $
\begin{tikzpicture}[scale=.8]
\node (b1) [small vert] at (0.5,2) {};
\node (c) [small vert] at (.5,1) {};
\node (d1) [small vert] at (.5,0) {};
\path [d] (b1) to (c);
\path [d,out=-115,in=125] (b1) to (d1);
\path [d] (c) to (d1);
\end{tikzpicture}
$\ $
\begin{tikzpicture}[scale=.8]
\node (b1) [small vert] at (0,2) {};
\node (b2) [small vert] at (1,2) {};
\node (c) [small vert] at (.5,1) {};
\node (d1) [small vert] at (.5,0) {};
\path [d] (b1) to (c);
\path [d] (b2) to (c);
\path [d,out=-90,in=125] (b1) to (d1);
\path [d] (c) to (d1);
\end{tikzpicture}

\caption{
\label{fig:small ruleset graphs 1}
Simple rulegraphs up to isomorphism with formal birthday at most $2$. 
}
\end{figure}

Consider a rulegraph $\mathsf{R}$ with formal birthday $d=\fbd(\mathsf{R})$. A position $p$ of $\mathsf{R}$ is called a \emph{top} position if $\fbd(p)=d$. Otherwise $p$ is called a \emph{rest} position. The \emph{index} of  $\mathsf{R}$ is $(t,u)$, where $t$ is the number of top positions and $u$ is the number of rest positions in $\mathsf{R}$.

A simple rulegraph $\mathsf{R}$ with formal birthday $d+1$ can be built from a simple rulegraph $\mathsf{S}$ with formal birthday $d$ by adding some new positions that have formal birthday $d+1$. To ensure this, the option sets of these new positions must contain at least one top position of $\mathsf{S}$. The process is demonstrated in Example~\ref{ega:lev2}. To count the number of ways this can be done, we will partition the set of simple rulegraphs into classes based on their indices. 

Let $x_{d,t,u}$ be the number of simple rulegraphs with formal birthday $d$ and index $(t,u)$ up to isomorphism. Define 
\[
I_d:=\{(t,u) \mid x_{d,t,u}>0\}
\]
to be the set of possible indices for simple rulegraphs with formal birthday $d$. It is clear that 
\[
x_d=\sum_{(t,u)\in I_d} x_{d,t,u}.
\]

\begin{example}
\label{eg:x-values}
It is easy to see using Figure~\ref{fig:small ruleset graphs 1} that $I_0=\{(1,0)\}$, $I_1=\{(1,1)\}$, and $I_2=\{(1,2),(2,2)\}$. We have $x_{0,1,0}=1$, $x_{1,1,1}=1$, $x_{2,1,2}=2$, and $x_{2,2,2}=1$. 
\end{example}

Define $\tilde{\mathsf{R}}$ to be the rulegraph obtained by removing all top positions from $\mathsf{R}$.

\begin{proposition}
\label{prop:Indices}
If $d\ge 0$, then 
\[
I_{d+1}=\{(T,t+u)\mid (t,u)\in I_d, 1\le T\le 2^{t+u}-2^u\}.
\]
\end{proposition}

\begin{proof}
A pair $(T,U)$ is in $I_{d+1}$ exactly when there is a simple rulegraph $\mathsf{R}$ with formal birthday $d+1$ and index $(T,U)$. Let the simple rulegraph $\tilde{\mathsf{R}}$ have index $(t,u)$. Note that $t+u=U$ and $(t,u)\in I_{d}$. Let $\mathcal{O}$ be the family of subsets of $V(\tilde{\mathsf{R}})$ that contain at least one top position. It is clear that $|\mathcal{O}|=2^{t+u}-2^u$. The option sets of the $T$ top positions of $\mathsf{R}$  must be distinct elements of $\mathcal{O}$. So we must have $1\le T\le |\mathcal{O}|$. 
\end{proof}

\begin{example}
\label{ega:lev2}
Consider the simple rulegraph $\mathsf{R}$ consisting of a single arrow $(b,a)$ with formal birthday 1 and index $(t,u)=(1,1)$. 
The number of sets of positions that contain at least one top position is $2^{t+u}-2^u=2^2-2^1=2$. These subsets are in $\mathcal{O}:=\{\{b\},\{a,b\}\}$. The third and fourth rulegraphs of Figure~\ref{fig:small ruleset graphs 1} are built from $\mathsf{R}$ by adding $T=1$ new position $p$ satisfying $\Opt(p)\in\mathcal{O}$. This can be done in $\binom{|\mathcal{O}|}{T}=\binom{2}{1}=2$ ways. The fifth rulegraph of Figure~\ref{fig:small ruleset graphs 1} is also built from $\mathsf{R}$ by adding $T=2$ new positions $p$ and $q$ satisfying $\Opt(p),\Opt(q)\in\mathcal{O}$ and $\Opt(p)\ne\Opt(q)$. This can be done in $\binom{|\mathcal{O}|}{T}=\binom{2}{2}=1$ way.  
\end{example}



Define  $I_{d}^U:=\{(t,u)\in I_{d} \mid t+u=U\}$.

\begin{proposition}
If $d\ge 0$ and $(T,U)\in I_{d+1}$, then 
\[
x_{d+1,T,U}=\sum_{(t,u)\in I_{d}^U} \binom{2^{t+u}-2^u}{T} x_{d,t,u}.
\]
\end{proposition}

\begin{proof}
If a simple rulegraph $\mathsf{R}$ has formal birthday $d+1$ and index $(T,U)$, then $\tilde{\mathsf{R}}$ has formal birthday $d$ and index $(t,u)$ for some $(t,u)\in I_{d}^U$. 

Now consider a simple rulegraph $\mathsf{S}$ with formal birthday $d$ and index $(t,u)\in I_{d}^U$. There are $x_{d,t,u}$ such simple rulegraphs. The family $\mathcal{O}$ of sets containing positions of $\mathsf{S}$ that have at least one top position has $2^{t+u}-2^u$ elements.
The number of simple rulegraphs $\mathsf{R}$ with formal birthday $d+1$ that we can build from $\mathsf{S}$ by adding $T$ new top positions with formal birthday $d+1$ is $\binom{{2^{t+u}-2^u} }{T}$ since each of the $T$ top positions of $\mathsf{R}$ must have a different option set that is an element of $\mathcal{O}$. 
\end{proof}

\begin{remark}
\label{rem:alt xd}
The computation of $x_d$ can be sped up using the alternate formula
\[
x_d=\sum_{(t,u)\in I_{d-1}} 
\big(2^{2^{t+u}-2^u}-1\big)x_{d-1,t,u},
\]
which has fewer terms.
The proof is similar to that of the previous proposition. If  $\mathsf{R}$ is a simple rulegraph with  formal birthday $d$, then  $\tilde{\mathsf{R}}$ has formal birthday $d-1$ and index $(t,u)$ for some $(t,u)\in I_{d-1}$. For any simple rulegraph $\mathsf{S}$ with formal birthday $d-1$ and index $(t,u)\in I_{d-1}$, we can build a simple rulegraph with formal birthday $d$ by adding an arbitrary positive number of top positions, as long as they all have mutually different option sets that are all  elements of $\mathcal O$ (defined in the previous proof). Since the number of nonempty subsets of $\mathcal O$ equals $2^{2^{t+u}-2^u}-1$, this is the number of rulegraphs that can be built from $\mathsf{S}$. The formula now follows since there are $x_{d-1,t,u}$ choices for the rulegraph $\mathsf{S}$.
\end{remark}

\begin{example}
We computed the values of $x_d$ for $d\in\{0,1,2\}$ in Example~\ref{eg:small_ruleset}. 
Let us calculate $x_3$. 
We saw in Example~\ref{eg:x-values} that $I_2=\{(1,2),(2,2)\}$, and $x_{2,1,2}=2$, $x_{2,2,2}=1$.
Remark \ref{rem:alt xd} implies that
\[
x_3=\big(2^{2^{3}-2^2}-1\big)x_{2,1,2}+\big(2^{2^{4}-2^2}-1\big)x_{2,2,2}=15\cdot 2+4095\cdot 1=4125.
\]
The sequence $(x_d)$ grows fast. The same formula also allows us to calculate the next term, which has 19724 digits. Further terms are computationally impractical. The sequence does not appear in~\cite{oeis}.
\end{example}

\begin{proposition}
\label{prop:difNonIso}
If the gamegraphs $\mathsf{R}_p$ and $\mathsf{R}_q$ are isomorphic  for some positions $p$ and $q$ of a simple rulegraph $\mathsf{R}$, then $\mathsf{R}_p=\mathsf{R}_q$.
\end{proposition}

\begin{proof}
For a contradiction suppose that $p\ne q$. Let $\alpha:\mathsf{R}_p\to\mathsf{R}_q$ be an isomorphism. Then $\alpha$ is also option preserving and $\alpha(p)=q$. The unique terminal position $t$ of $\mathsf{R}$ satisfies $\alpha(t)=t$. Let $d$ be the smallest formal birthday for which there is a position $r$ such that $\fbd(r)=d$ and $\alpha(r)\ne r$. Note that $\fbd(s)<d$ for all $s\in\Opt(r)$. Hence $\Opt(\alpha(r))=\alpha(\Opt(r))=\Opt(r)$, which is a contradiction since $\mathsf{R}$ is simple.
\end{proof}

\begin{example}
\label{exa:Md}
Define the rulegraph $\mathsf{M}^d$, with position set $\mathcal P^d(\{\emptyset\})$, where $\mathcal P^d$ is the power set operator $\mathcal P$ composed with itself $d$ times, and the pair $(A,B)$ is an arrow if and only if $B\in A$. Notice that $\mathsf{M}^d$ has ${}^d2$ positions and $\fbd(\mathsf{M}^d)=d$.
The rightmost graph of Figure~\ref{fig:small ruleset graphs 1} is isomorphic to $\mathsf{M}^2$.
\end{example}

\begin{proposition}
There exists a simple rulegraph with $m$ positions and formal birthday $d$ if and only if $d+1\leq m \leq {}^d2$.
\end{proposition}

\begin{proof}
Assume that $\fbd(\mathsf{R})=d$ for a simple rulegraph $\mathsf{R}$ with $m$ positions.  It is clear that $d+1\leq m$.  
If $p$ is a position of $\mathsf{R}$, then $\fbd(\mathsf{R}_p)\le d$. 
Proposition~\ref{prop:difNonIso} implies that $\Gamma(\mathsf{R})$ consists of $m$ non-isomorphic gamegraphs. Hence $m=|\Gamma(\mathsf{R})|\leq {}^d2$ by 
Proposition~\ref{prop:number of rulegraphs}.

To show the other implication, let $p_0:=\emptyset$ and $p_{k}:=\{p_{k-1}\}$ for $1\le k\leq d$. Define $P^d:=\{p_0,\ldots,p_d\}\subseteq V(\mathsf{M}^d)$. Suppose $d+1 \leq k < {}^d2$.  We can construct a desired simple rulegraph $\mathsf{R}$ with $k$ positions and $\fbd(\mathsf{R})=d$ from $\mathsf{M}^d$ by removing  positions from $V(\mathsf{M}^d)\setminus P^d$ one by one, making sure that we never remove a position of formal birthday $b-1$ until we have removed all positions of $V(\mathsf{M}^d)\setminus P^d$ of formal birthday $b$.
\end{proof}

\begin{proposition}
For each $d$, there exists a unique simple rulegraph with ${}^d2$ positions and formal birthday $d$ up to isomorphism.
\end{proposition}

\begin{proof}
The rulegraph $\mathsf{M}^d$ from Example~\ref{exa:Md} has ${}^d2$ positions and formal birthday $d$. To see that it is unique (up to isomorphism), note that for any rulegraph $\mathsf{R}$ from the proposition statement, the set $\Gamma(\mathsf{R})$ is a set of ${}^d2$ nonisomorphic gamegraphs with formal birthday at most $d$. As this equals the total number of such gamegraphs, this set is uniquely determined. As the arrows between positions corresponding to different gamegraphs are also uniquely determined, we have what was to be shown.
\end{proof}

\begin{proposition}
There are $2^{\binom{d+1}{2}-d}$ simple rulegraphs  with $d+1$ positions and formal birthday $d$ up to isomorphism.
\end{proposition}

\begin{proof}
To construct such a rulegraph, we start with a directed path of length $d$. 
The inclusion of any selection of the remaining $\binom{d+1}{2}-d$ possible arrows produces a different rulegraph that satisfies the requirements.
\end{proof}

Note that all of the rulegraphs described in the previous result are actually gamegraphs.

\subsection{Counting by the number of positions}

Next we enumerate simple rulegraphs and gamegraphs by the number of positions. 
A set $S$ is called \emph{full} or \emph{transitive} if every element of $S$ is also a subset of $S$. The number $\hat{e}_{n}$ of full sets with $n$ elements is featured as sequence \cite[A001192]{oeis}. The number of extensional acyclic digraphs with $n$ vertices is also $\hat{e}_{n}$, as shown in \cite{extensionalAcyclic}. The number $\hat{e}_{n,s}$ of acyclic extensional digraphs with with $n$ vertices and $s$ sources is studied in \cite{DigraphParameters}.

\begin{proposition}
The numbers of simple rulegraphs and simple gamegraphs with $n$ vertices are $\hat{e}_{n}$ and $\hat{e}_{n,1}$, respectively. 
\end{proposition}

\begin{proof}
Proposition~\ref{prop:slim iff injective} implies that simple rulegraphs are exactly the acyclic extensional digraphs.
\end{proof}

Table~\ref{cvert} shows some values of $\hat{e}_{n}$ and $\hat{e}_{n,1}$ as computed in \cite{extensionalAcyclic}.

\begin{table}[ht]
\renewcommand{\arraystretch}{1.3}
\[
\begin{array}{cccccccccc}
\hline
n&1&2&3&4&5&6&7&8&9  \\\hline
\text{rulegraphs }\hat{e}_{n}&1& 1&2& 9&88& 1802& 75598&6{,}421{,}599& 1{,}097{,}780{,}312\\
\text{gamegraphs } \hat{e}_{n,1}&1& 1&2& 8& 68& 1248& 48640&3{,}944{,}336&655{,}539{,}168\\
\hline
\end{array}
\]
\caption{The number of simple rulegraphs and gamegraphs with $n$ vertices.}
\label{cvert}
\end{table}

\section{Further questions}

We finish with a list of open problems.
\begin{enumerate}
\item 
In light of Proposition~\ref{prop:sourceIffSurj} and Example~\ref{ex:option preserving not surjective}, what hypotheses on rulegraphs are equivalent to surjectivity for option preserving maps?
\item Example~\ref{ex:mouse} shows that the gamegraph model allows outcome functions that are not constant on the terminal positions.  Studying these outcome functions might be fruitful.
\item  Are there analogs to the other isomorphism theorems from universal algebra? 
\item The connection between rulegraphs and collections of gamegraphs could be explored more deeply. What is the ``best'' rulegraph for a given collection of gamegraphs? What does ``best'' even mean?

\item Can we characterize the lattices that we get as $\Con(\mathsf{R})$? 
\item  If we model the partizan situation using edge-colored digraphs, can we obtain analogous results?
\item  Can our results be extended for loopy games? The development seems to be more complex. For example, the $\Opt$ map is injective on the directed cycle with three vertices but this loopy game is not simple because the map that takes each vertex to the single vertex of a loop is option preserving. It is possible that the approach described in Remark~\ref{rem:altBowtie} is needed. Similar work was done in \cite[Section 1.3]{SiegelThesis}.
\end{enumerate}

\section*{Acknowledgments}
The authors would like to thank the organizers of Combinatorial Game Theory Colloquium IV, where this collaboration began. The first author was supported by the Ministry of Science, Technological Development and Innovation of Serbia (grant no.\@ 451-03-47/2023-01/200125), and the fourth author was supported by the same Ministry through Mathematical Institute of the Serbian Academy of Sciences and Arts.

\bibliographystyle{plain}
\bibliography{game}

\end{document}